\documentclass[12pt]{amsart}

\usepackage{amssymb, stmaryrd}
\usepackage{hyperref}
\usepackage[margin=.85 in]{geometry}
\usepackage{color,soul}
\usepackage{multicol}
\usepackage[table]{xcolor}
\usepackage{enumerate}
\usepackage{mdframed}
\usepackage[T1]{fontenc}
\usepackage{tikz}
\usetikzlibrary{positioning}
\usepackage{multicol}
\usepackage{caption}
\usepackage{subcaption}
\usepackage{float}
\usepackage{comment}

\makeatletter
\def\bbordermatrix#1{\begingroup \m@th
  \@tempdima 4.75\p@
  \setbox\z@\vbox{%
    \def\cr{\crcr\noalign{\kern2\p@\global\let\cr\endline}}%
    \ialign{$##$\hfil\kern2\p@\kern\@tempdima&\thinspace\hfil$##$\hfil
      &&\quad\hfil$##$\hfil\crcr
      \omit\strut\hfil\crcr\noalign{\kern-\baselineskip}%
      #1\crcr\omit\strut\cr}}%
  \setbox\tw@\vbox{\unvcopy\z@\global\setbox\@ne\lastbox}%
  \setbox\tw@\hbox{\unhbox\@ne\unskip\global\setbox\@ne\lastbox}%
  \setbox\tw@\hbox{$\kern\wd\@ne\kern-\@tempdima\left[\kern-\wd\@ne
    \global\setbox\@ne\vbox{\box\@ne\kern2\p@}%
    \vcenter{\kern-\ht\@ne\unvbox\z@\kern-\baselineskip}\,\right]$}%
  \null\;\vbox{\kern\ht\@ne\box\tw@}\endgroup}
\makeatother

\def\VR{\kern-\arraycolsep\strut\vrule &\kern-\arraycolsep}
\def\vr{\kern-\arraycolsep & \kern-\arraycolsep}

\theoremstyle{plain}
\newtheorem{theorem}[subsubsection]{Theorem}
\newtheorem{proposition}[subsubsection]{Proposition}
\newtheorem{lemma}[subsubsection]{Lemma}

\newtheorem{corollary}[subsubsection]{Corollary}

\theoremstyle{definition}
\newtheorem{definition}[subsubsection]{Definition}

\newtheorem{example}[subsubsection]{Example}

\newtheorem{conjecture}[subsubsection]{Conjecture}

\newtheorem{remark}[subsubsection]{Remark}

\normalfont
\usepackage[T1]{fontenc}{}

\newcommand{\fermat}{x^{q+1} + y^{q+1} + z^{q+1} + w^{q+1}}

\DeclareMathOperator{\GL}{GL}
\DeclareMathOperator{\PGL}{PGL}

\DeclareMathOperator{\U}{U}
\DeclareMathOperator{\PU}{PU}

\DeclareMathOperator{\Aut}{Aut}
\DeclareMathOperator{\PG}{PG}

\date{\today}

\title{Geometry of Smooth Extremal Surfaces}
\author{Anna Brosowsky, Janet Page, Tim Ryan and Karen E. Smith}
\email{ kesmith@umich.edu}
\thanks{Anna Brosowsky was partially funded by NSF DMS \#1840234 and  \#2101075, 
Karen Smith was partially funded  by NSF DMS \#1801697 and \#2101075.}

\begin{document}

\abstract
We study the geometry of the smooth projective surfaces that are defined by Frobenius forms, a class of homogenous polynomials  in prime characteristic recently shown to have minimal possible F-pure threshold among forms of the same degree.
We call these surfaces {\it  extremal surfaces}, and show that their geometry is reminiscent of the geometry of smooth cubic surfaces, especially non-Frobenius split cubic surfaces of characteristic two, which are examples of extremal surfaces. For example, we show that an extremal surface $X$ contains $d^2(d^2-3d+3)$ lines where $d$ is the degree, which is notable since the number of lines on a complex surface is bounded above by a quadratic function in $d$. Whenever two of those lines meet, they determine a $d$-tangent plane to $X$ which consists of a union of $d$ lines meeting in one point; we count the precise number of such "star  points" on $X$, showing that it is quintic in the degree,  which recovers the fact that there are exactly 45 Eckardt points on an extremal cubic surface.
 Finally, we generalize the classical notion of a double six for cubic surfaces to a double $2d$ on an extremal surface of degree $d$. We show that, asymptotically in $d$,  smooth extremal surfaces have at least $\frac{1}{16}d^{14}$ double $2d$'s. A key element of the proofs is using the large automorphism group of extremal surfaces which we show acts transitively on many sets, such as the set of (triples of skew) lines on the extremal surface.
Extremal surfaces are closely related to finite Hermitian geometries, which we recover as the 
$\mathbb F_{q^2}$-rational  points of special extremal surfaces defined by Hermitian forms over $\mathbb F_{q^2}$.}

\maketitle


\section{Introduction}

Let $k$ be an algebraically closed field of positive characteristic $p$. Our goal is to study the geometry of smooth {extremal surfaces}  over $k$.
 
 An {\bf extremal surface} is a  surface in $\mathbb P^3$ defined by a homogenous polynomial of {\it smallest possible F-pure-threshold} among reduced forms of the same degree. 
While not obvious such  forms exist,  a sharp lower bound on the F-pure threshold in terms of degree was proved in \cite[1.1]{extremal}, where  the forms achieving it were classified and dubbed {\bf Frobenius forms}.    
 The F-pure threshold  is a measurement of singularities\footnote{{The F-pure threshold was first 
 defined as a "characteristic $p$ analog" of the log canonical threshold, a well-known invariant of complex singularities, 
 by Takagi and Watanabe  \cite{takagi+watanabe.F-pure_thresholds}, who were building on the work of Hara and Yoshida \cite{HaraYoshida}. See also      \cite{mustata+takagi+watanabe.F-thresholds}, \cite{blickle+mustata+smith.discr_rat_FPTs} or \cite{benito+faber+smith.measuring_sing_with_frob}.}
}
 with smaller thresholds representing "worse singularities," so forms with minimal F-pure threshold cut out 
 "maximally singular" cones in affine space. Thus it is natural to expect the corresponding projective hypersurfaces to exhibit some extremal geometric properties as well.

  The simplest case of an extremal surface is a non-Frobenius split cubic surface of characteristic two, which were studied in depth in \cite{cubics}.  Geometrically, extremal cubic surfaces can be characterized among all cubic surfaces as those that  {\it  admit no triangles}. 
 To understand this statement, recall that each smooth cubic surface admits exactly forty-five plane sections consisting of a union of three lines,  typically forming a "triangle". Some special cubic surfaces  admit one or more such { tri-tangent sections} in which the three lines meet at some point (called an Eckardt point).  An extremal cubic surface has the highly unusual property that {\it each and every one} of the forty-five  tri-tangent plane sections  consists of three concurrent lines. Such "triangle-free" cubic surfaces do not exist over $\mathbb C$ nor indeed over any field of odd characteristic. Extremal cubic surfaces exist only  in characteristic two, precisely when the cubic form cutting out the surface is {a Frobenius form.}  These results are all worked out in \cite{cubics}; see also \cite{extremal},  \cite{dolgachev-duncan.automorphisms}, \cite[5.5]{hara.rational-singularities}, \cite[1.1]{homma} and \cite[20.2]{Hirschfeld} for related work.
 
This paper explores ways in which smooth projective surfaces defined by Frobenius forms have "extremal" geometric features analogous to the abundance of concurrent configurations of  lines on a smooth  extremal cubic surface. 
  Like extremal cubics, higher degree extremal surfaces contain no triangles: if two lines on a smooth extremal surface  $X$ meet at some point $p$, then the tangent plane section $T_pX\cap X$  at $p$ consists of $q+1$  distinct lines meeting at $p$ (Corollary~\ref{noTriangle}). We call such a configuration of lines on an extremal surface a {\bf star} and the point of concurrency a {\bf star point};  note that a star point is precisely an {Eckardt point} in the case of cubic surfaces. 
     
   We prove that smooth extremal surfaces have a  large number of
stars---indeed the number of stars grows like $d^5$ where $d$ is the degree of the surface---which means that extremal surfaces have a large number of lines---exactly $d^2(d^2-3d+3)$ to be precise (see Corollary~\ref{count}). This  contrasts  sharply with the characteristic zero case, where  the number of lines on a smooth  surface in $\mathbb P^3$ is bounded above by a quadratic function in the degree; see \cite{Segre.43, Rams+Schutt.15-64lines} or \cite{Boissiere+Sarti}. Bauer and Rams recently showed that a quadratic bound holds even in characteristic $p$, provided $p>d$ \cite{Bauer+Rams}. 
On the other hand,  the quadratic bound has been known to be false in non-zero characteristic (see {\it e.g.} \cite{Rams+Schutt.15-112lines}).
 Corollary~\ref{count} confirms that it is wildly false in every characteristic, even for 
  $d=p+1$.
   
The main theme of this paper is that extremal surfaces exhibit fascinating geometry  reminiscent of the geometry of lines on cubic surfaces. Most substantially, in Section~6, we  generalize the classical notion of a "Double Six" on a cubic surface, showing that an extremal surface of degree $d$  admits many configurations of  "Double 2d"'s (Theorem~\ref{Quadric2}). 
 To construct and establish the abundance of such Double 2d's, we use pairs of {\it quadric configurations,} a concept investigated in Section 5.  A quadric configuration is a collection  of $2d$  lines on a surface of degree $d$ all lying on the same quadric. While most surfaces do not contain any quadric configuration, we show that, like cubic surfaces, a degree $d$ extremal surface contains many quadric configurations---roughly $\frac{d^9}{2}$ for large $d$ (Corollary~\ref{QuadricCount}).
In Conjecture~\ref{DoubleConj}, we speculate   that, as  is classically known for  cubic surfaces, every double $2d$ on an extremal surface is a union of two quadric configurations. In Theorem~\ref{ProgressConj} we
   prove this conjecture for $d>10$ and $d<5$.  
   
  A key tool used throughout is that  the symmetry group of an extremal surface is quite large. Indeed, we prove it acts   transitively 
 on several large  sets, including  the set of all stars
  (Proposition~\ref{transStars}) 
and the set of all  ordered triples of skew lines on $X$   (Theorem~\ref{TripleTrans}). We find an explicit description of the automorphism group of an extremal surface (Theorem~\ref{autogroup}) that recovers a theorem  of Shioda on the automorphism group of certain Fermat hypersurfaces, and of Duncan and Dolgachev for  cubics "with no canonical point" (see Remark~\ref{CompareShioda}).

Extremal varieties are closely connected to finite Hermitian geometry, although our approach is completely independent (see \cite{BC, Hirschfeld}). Indeed, a {\it Hermitian form} is a (very) special type of Frobenius form defined over $\mathbb F_{q^2}$; see \S~\ref{Herm}.
For this reason, several  basic facts established in Sections 3 and 4 will  sound familiar to experts in finite geometry where Hermitian forms over $\mathbb F_{q^2}$ play a starring role. For example, our counts of star points and lines on an extremal  surface (Corollary~\ref{count}) produce well-known numbers of points and lines in a Hermitian sub-geometry  of the finite projective geometry $\PG(3, q^2)$.  
Proposition~\ref{HermStar} explains why: we show that
{\it if the extremal variety happens to be defined by a Hermitian form over} $\mathbb F_{q^2}$, then
 its  star points are simply its $\mathbb F_{q^2}$-points. Our paper is independent of  the vast  theory of finite geometries ({\it e.g.} \cite{Hirschfeld}), and uses only standard algebraic geometry over an algebraically closed field  as one might find in a text such as Shafarevich \cite{Shafarevich}; indeed, we discovered the connection with finite geometry only after our work was complete. None-the-less, we include self-contained proofs of  some results, such as the structure of the automorphism group of an extremal variety (Theorem~\ref{autogroup}), which could, {\it a postieri}, be deduced from well-known results in finite geometry.

 Our work
 connects extremal varieties to a diverse array of active research groups throughout pure and applied mathematics including  in coding and design theory \cite{Goppa, Etzion+Storme, Tsfasman+Vladut+Zink},  
 rational points on curves and varieties \cite{HommaKim, Hirschfeld+Korchmaros+Torres}, graph theory \cite{Faina+Korchmaros}, cryptology \cite{Klein+Storme},  group theory \cite{Grove, Tits.76}, and the combinatorics of hyperspace arrangements and generalized quadrangles \cite{Payne+Thas.09}.  Nearly all this research  is written from a dramatically different perspective from our paper. We hope to inspire algebraic geometers to investigate some of the many open problems, for example, in  \cite{Hirschfeld+Thas.15}, and conversely, help researchers in diverse fields gain  access to new techniques. A small sample of related literature includes
 \cite{Segre.67}, \cite{Hirschfeld+Thas.16},
    \cite{Tsfasman+Vladut+Nogin}, \cite{Tits.79},  \cite{VanMaldeghem}, \cite[\S~35]{Kollar15}, and the references therein. 
    
    \noindent
    {\bf Acknowledgements.} This work is an offshoot of a project begun at a Banff workshop called {\it Women in Commutative Algebra}, which produced the papers \cite{cubics}, \cite{WICA},  and \cite{extremal}.  We would like to acknowledge the valuable discussions with the other participants in those earlier projects: Elo\'isa Grifo, Zhibek Kadyrsizova, Jennifer Kenkel, Jyoti Singh, Adela Vraciu, and Emily Witt. We are also grateful to J\'anos Koll\'ar for suggesting the connection with Hermitian geometry.
  \section{Basics of Frobenius forms}\label{basics}
  
This section consolidates   needed known facts and terminology  about Frobenius forms.

Fix a field $k$ of positive characteristic $p$, and let $q$ denote  $p^e$ for   some fixed   positive integer  $e$. 
A Frobenius form (in $n$ variables, say) is a homogeneous polynomial of degree $p^e+1$  in the "Frobenius power" $\langle x_1^{p^e}, x_2^{p^e}, \dots, x_n^{p^e}\rangle$ of the unique homogenous maximal ideal of the polynomial ring.  Put differently, a Frobenius form is a polynomial  $h$ that can be written  
 $ \sum_{i=1}^n x_i^{q} L_i$, where  $L_i$ are linear forms. In particular, every Frobenius form admits a matrix factorization  
\begin{equation}\label{matrix}
 h = \begin{bmatrix}   x_1^{q} &  x_2^{q}  & \hdots & x_n^{q}\end{bmatrix}
 A \begin{bmatrix}  x_1 \\ x_2 \\ \vdots  \\ x_n \end{bmatrix} =  ( \vec{x}^{[q]})^{\top} A \,  \vec{x},
 \end{equation}
  where $A$ is the  unique  $n\times n$ matrix whose $i$-th row is made up of the coefficients of the linear form $L_i$.  Here,  for a matrix $B$ of any size,  the notation $B^{[q]}$ denotes the matrix obtained by raising all entries to the $p^e$-th power,  and $B^\top$ denotes the transpose of $B$. The notation  $\vec x$ denotes a column vector of   size   $n$.

 \subsection {Changes of Coordinates} \label{coordinateChange}  The set of Frobenius form is preserved by  arbitrary linear changes of coordinates, since both degree and the ideal $\langle x_1^{q}, x_2^q, \dots, x_n^q\rangle$ are preserved.
  If $g$ is an invertible $n\times n$ matrix representing some linear change of coordinates, 
   and the Frobenius form $F$ is represented by the matrix $A$, then  the Frobenius form $g^*F$ obtained after changing coordinates is 
  represented by the matrix   \begin{equation}\label{action}
  \left[g^{[p^e]}\right]^{\top}  A g.
  \end{equation}
   See 
  \cite[\S~5]{extremal} for details.

A Frobenius form is said to be  {\bf non-degenerate} if it cannot be written as a polynomial in fewer variables after any linear change of coordinates.

  The {\bf rank} of a Frobenius form is the rank of the representing matrix. The rank is the same as the codimension of the singular locus of the corresponding hypersurface \cite[5.3]{extremal}.

\begin{theorem}\label{smooth} \cite[6.1]{extremal}\cite{beauville}.
All maximal rank Frobenius forms of fixed degree and number of variables over a fixed algebraically closed field $k$ are the same up to linear change of variables.
\end{theorem}
Theorem~\ref{smooth} says there is a unique {\it smooth\/} extremal hypersurface of each dimension and allowable degree. 
In particular, a {\it smooth extremal surface}  in $\mathbb P^3$  can be assumed, after suitable choice of projective coordinates, to be defined by $x^{q+1}+y^{q+1}+z^{q+1}+w^{q+1}$ or, equivalently, by  $x^{q}w+  w^qx + y^{q}z  + z^{q}y$ or any full rank Frobenius form in $x, y, z, w$.

More generally, 
 Frobenius forms are classified up to linear changes of coordinates over an algebraically closed 
field. For each $n$, the number of distinct projective equivalence classes of Frobenius forms of a fixed degree $q+1$ 
 is equal to the number of partitions of $n$. See \cite[7.1]{extremal} for the precise statement.

\begin{example}\label{3var} There are three equivalence classes of  non-degenerate Frobenius forms in three variables and of degree $q+1$, corresponding, respectively, to the three matrices 
$$
\left[
\arraycolsep=3.3pt
\begin{array}{ccc}
\cellcolor{blue!15}1 & 0 & 0 \\ 
0 & \cellcolor{blue!15}1 & 0 \\ 
0 & 0 & \cellcolor{blue!15}1 
\end{array}
\right],
\,\,\,\,\,\,\,\,
\left[
\arraycolsep=3.3pt
\begin{array}{ccc}
\cellcolor{blue!15} 1 & 
0 & 
0 \\ 
0 & 
\cellcolor{blue!15} 0 &
\cellcolor{blue!15} 1 \\ 
0 & 
\cellcolor{blue!15} 0 & 
\cellcolor{blue!15} 0 
\end{array}
\right]
\,\,\,\,\,\,{\text{and}} \,\,\,\,\,\,\,\,
\left[
\arraycolsep=3.3pt
\begin{array}{>{\columncolor{blue!15}} c >{\columncolor{blue!15}} c >{\columncolor{blue!15}} c}
0 & 1 & 0 \\ 
0 & 0 & 1 \\ 
0 & 0 & 0 
\end{array}
\right].
$$

\noindent
 These determine, respectively, the forms $x^{q+1} + y^{q+1} + z^{q+1}, \,\, x^{q+1} + y^qz\, $ and  $\,x^qy+ y^qz$.  
\end{example}

\subsection{Stars}  Smooth extremal surfaces have distinguished points with special geometry analogous to the Eckardt points on a cubic surface:

\begin{definition}\label{starDef}
 A {\bf star} on  a smooth surface $X$ of degree $d$ is a  configuration of $d$ lines on $X$, all meeting at one point $p$ called the {\bf center} of the star, or a {\bf star point.}
 \end{definition}

 If $L$ is a line on a smooth surface $X$ and $p$ is a point on $L$, then $L\subset T_pX$, the tangent plane to $X$ at $p$. Thus the $d$ lines forming a star on $X$ are coplanar---all lie in $T_pX$ where $p$ is the center of the star. In this case, the plane section $T_p X \cap X$ is the reduced union of the $d$ lines of the star.  A plane containing a star is called a {\bf star plane.}
 Star planes are uniquely determined by their centers and vice versa, since  each star plane is the tangent plane  to $X$ at the center of its star. Stars are defined and studied for higher dimensional hypersurfaces in \cite{Cools+Coppens}.

\begin{example}\label{NotFermat} Consider the extremal surface $X$ defined by $x^{q}w+ w^qx +  y^{q+1}+z^{q+1}$.
Intersecting with the plane $H$  defined by  $w$, we see a star $X\cap H$ consisting of  $q+1$ distinct lines 
$$\{\mathbb V(w, y-\nu z)\; | \;  \nu^{q+1} = -1\},
$$ all  intersecting in the star point
 $p = [1:0:0:0]$.  Thus $H$ is a star plane with center $p$.
\end{example}

\begin{remark}\label{StarLineSymmetry}  The lines in Example~\ref{NotFermat}  are indistinguishable up to projective transformation. Indeed,  the projective linear changes of coordinates $[x:y:z:w]\mapsto [x: y: \mu z:w]$ (for each $\mu \in \mu_{q+1}$) stabilizes  the surface $X$ and its star plane $H$ while  transitively permuting around the lines  in $H$.
\end{remark}


\subsection{Plane Sections of Extremal Surfaces}

 \begin{proposition}\label{planeSection} \cite{cubics}
 A plane section of a smooth extremal surface is one of the following types of divisors, all defined by Frobenius forms:
 \begin{enumerate}
 \item A smooth extremal curve.
  \item A singular extremal curve with an isolated  cuspidal singularity.
 \item The reduced sum of a line and an irreducible curve  tangent at  one point.
 \item A {\bf star} of lines on the surface.
  \end{enumerate}
  In particular, the plane section $H\cap X$ is a star if and only if  the Frobenius form  $\overline{F}$ defining $X\cap H$ in $H$  is  degenerate---that is, if and only if  $\overline{F}$ can be written  as a Frobenius form in two (of three) homogeneous coordinates for  the projective plane $H$.
 \end{proposition}

 \begin{proof}
 Any plane section of an extremal surface $X$ is extremal \cite[8.1]{extremal}, so to understand a plane section $X\cap H$,  we look at the classification of  Frobenius forms in three variables given in \cite[7.1]{extremal}. 
 The non-degenerate ones are in one-one correspondence with the three partitions of 3: these are described in  Example~\ref{3var} (up to projective change of coordinates) and produce  the first three types of divisors listed above. It is also possible that 
 $X\cap H$ is defined by a degenerate Frobenius form. These are classified by partitions of 2 and of 1: 
 \begin{enumerate}
 \item[(i)] A star, projectively equivalent to  $x^qy + y^qx$ 
 \item[(ii)] the non-reduced scheme projectively equivalent to $x^qy$
 \item[(iii)] the non-reduced scheme  projectively equivalent to $x^{q+1}.$
 \end{enumerate}
 However, because $X$ is smooth, the plane sections $X\cap H$ are reduced (by e.g. \cite[1.15]{Zak}) and so only the first of these possibilities occurs. 
 \end{proof}

Proposition~\ref{planeSection}
has  the following  consequences; we prove only the second as the first is immediate;

 \begin{corollary}\cite[8.11]{extremal}
  \label{noTriangle}
Any collection of coplanar lines on an extremal surface is concurrent. In particular, an extremal surface contains no triangles.
  \end{corollary}

 \begin{corollary}\label{LineInStar}
 Every line on a smooth extremal surface is in some star.
 \end{corollary}
 
 \begin{proof}[Proof of Corollary~\ref{LineInStar}]
Fix a line $L$ on a smooth extremal surface $X$. There is {\it some} plane $H$ such that $X\cap H$ is a star, as there is no loss of generality in assuming  $X$ as in Example~\ref{NotFermat} (Theorem~\ref{smooth}). 
Now, if $L$ lies in $H$, then the Corollary is proved, as  $L$ is a line in the star  $X\cap H$. But  if $L$ does not lie in $H$, then $L$  meets $H$ at some point $p'$.  Because point  $p' \in  X\cap H$, we know $p'$ lies on some line  $L'$ in the star   $X\cap H$.  Since
the two lines $L$ and $L'$  intersect at $p'$,  the plane $H'$  they span is a star plane centered at $p'$ (Proposition~\ref{planeSection}). The star  $X\cap H'$ thus contains our line $L$. 
 \end{proof}


\subsection{Extremal Collections of Points}

The  automorphism group of  any zero dimensional smooth extremal hypersurface $Y$
 (that is, of  a reduced extremal configuration of points in $\mathbb P^1$) acts transitively on the points in $Y$; this follows immediately from   Remark~\ref{StarLineSymmetry} by projectivizing the 
 the star plane.  
 However, a stronger symmetry holds that we record for future reference:

\begin{proposition}\label{dim0} Let $Y\subset \mathbb P^1$ be a reduced extremal configuration of points--that is, assume $Y$ is defined by a rank two Frobenius form in two variables. 
Then the projective linear automorphism group of $Y$ acts three-transitively on the points of $Y$. Furthermore, $\Aut(Y)$ is isomorphic to $\PGL(2, \mathbb F_q)$.
 \end{proposition}

\begin{proof}
We may assume that $Y$ is defined by the form $yz^q-zy^q$ (Theorem~\ref{smooth}), 
so $Y$ consists of the points $[\mu:1]$ where $\mu^q=\mu$, together with the "point at infinity" $[1:0]$.
 So the points of $Y$ are precisely the $\mathbb F_q$-points of $\mathbb P^1$. 

Now, given an ordered triple of three distinct points in $Y$, there is a unique automorphism $g$ of $\mathbb P^1$  sending them to any other ordered triple  in $Y$.
Because all six points  are defined over $\mathbb F_q$, the automorphism $g\in \PGL(2, k)$ is represented by a $2\times 2$ matrix with entries in $\mathbb F_q$, so that $g\in \PGL(2,  \mathbb F_q)$. In particular, it must send {\it every\/} $\mathbb F_q$-point of $\mathbb P^1$ to another $\mathbb F_q$-point of $\mathbb P^1$. That is, $g$ is an automorphism of $Y$. This establishes both claims of  Proposition~\ref{dim0}.
\end{proof}


\subsection{Hermitian Forms over Finite Fields}\label{Herm}
  A {\it Hermitian form} is a special kind of Frobenius form   in which the representing matrix $A$    satisfies $(A^{[q]})^\top=A$. In this case, all entries of $A$ satisfy  $a_{ij}^{q^2} = a_{ij}$, which means  they are in the finite field $\mathbb F_{q^2}$. Thus a Hermitian form is defined over the finite field $\mathbb F_{q^2}$. In this case, the Frobenius map ($x\mapsto x^q$) is an involution on the set of $\mathbb F_{q^2}$ points, so can  play a role analogous to complex conjugation.  See \cite[\S~19.1]{Hirschfeld}.
  
   The classification of Hermitian forms is well-known and simple: there is only one invariant, rank \cite[4.1]{BC}, where as the classification of Frobenius forms is more subtle \cite[7.1]{extremal}. On the other hand,  every {\it smooth} projective hypersurface defined by a Frobenius form is projectively equivalent (over the algebraically closed field $k$)  to one defined by a Hermitian form (Theorem~\ref{smooth}), although of course,  the needed change of coordinates is not usually defined over $\mathbb F_{q^2}$.

 \section{Configurations of Lines and Stars on Extremal Surfaces}

\subsection{Symmetry of Extremal surfaces}
 Smooth extremal surfaces are highly symmetric: 
 
\begin{proposition} \label{transStars} The automorphism group of a smooth extremal surface acts transitively on its set of stars. 
\end{proposition}

By {\bf automorphisms} here, we mean {\it projective linear} transformations of the surface in $\mathbb P^3$. Thus $\Aut(X)$ is a
 subgroup of   $\PGL(4, k)$; in Section~\ref{automorphisms}, this group is discussed in detail.

Before proving the proposition , we deduce a corollary:
\begin{corollary} \label{transLines} The automorphism group of a smooth extremal surface  $X$ acts transitively on the set of all pairs 
 $(H, L)$ consisting of a star plane $H$ and a  line $L$ in the star  $H\cap X$. In particular, $\Aut(X)$ acts transitively on the set of all lines on $X$.
 \end{corollary}

\begin{proof}[Proof of Corollary~\ref{transLines}] 
Without loss of generality, assume  $X$ is defined by $x^{q}w + w^qx + y^{q+1} + z^{q+1}$ (Theorem~\ref{smooth}).  Given an arbitrary pair $(H, L)$, Proposition~\ref{transStars} says there is an automorphism of $X$ taking $H$ to  the star plane $H'$  defined by $w=0$   (see also Example~\ref{NotFermat}). But then we can compose with an  automorphism of $X$ preserving  $H'$ while 
taking the image of $L$ to any line in  the star $H'\cap X$ (Remark~\ref{StarLineSymmetry}). \end{proof}

To prove Proposition~\ref{transStars}, we make use of the following lemma.

 \begin{lemma}\label{NiceForm}
 Given an arbitrary  star $X\cap H$ with center $p$ on a smooth  extremal surface $X$,  we may choose coordinates for  $\mathbb P^3$ so that 
\begin{equation}\label{useful}
 p=[0:0:0:1],  \,\, \,\,\,\, H = \mathbb V(x),  \,\,\,\,{\text{and}}  \,\,\,\,\,X = \mathbb V(x^q\ell + xw^q+y^qz+z^qy),
\end{equation}
for some linear form $\ell = ax+by+cz+w$.
 \end{lemma}
 
 \begin{proof}  Choose coordinates so  that the star plane $H$ is defined by $x=0$. In this case, the form $F$ defining $X$ is 
 $$
 F = xG + G'(y,z,w)
 $$
where $G$ is some form of degree $q$ and $G'$ is a Frobenius form in the variables $y, z, w$. The form $G'$ defines the star $X\cap H$ in the hyperplane $H\cong \mathbb P^2$. 
In particular,    $G'$  is  {\it degenerate}  (Proposition~\ref{planeSection}). So there is a linear change of coordinates involving only $y, z, w$ such that $G'$ transforms into a rank two Frobenius form in  two variables.  Since all reduced Frobenius forms in 2 variables are projectively equivalent, 
 we can assume  without loss of generality, that 
 \[ 
F =  xG +  yz^{q} + zy^q.
\]

   Observe that $xG\in \langle x^{q}, y^{q}, z^q, w^q \rangle$, which implies that   $G\in    \langle x^{q-1}, y^{q}, z^q, w^q \rangle$. Because $\deg G = q$, 
we can write
\[G = x^{q-1} \ell + (\alpha_1 y+\alpha_2 z + \alpha_3 w)^q
\]
for some scalars $\alpha_i$ and linear form $\ell$. That is,
$$
F =  x^q\ell  +  x(\alpha_1 y+\alpha_2 z + \alpha_3 w)^q +   zy^{q}+yz^{q}.
$$
The scalar $\alpha_3$ can not be zero. Indeed,   if $\alpha_3=0$,  then $F \in \langle x^q, y^q, z^q\rangle$, so  the rank of $F$ would be at most three and  $F$ could not define a smooth surface \cite[5.3]{extremal}.  So we may replace  the form $\alpha_1 y+\alpha_2 z + \alpha_3 w$ by $w$ (which changes $\ell$ but nothing else) to assume without loss of generality that 
\begin{equation}\label{d}
F = x^q\ell + xw^q +  zy^{q}+yz^{q}.
\end{equation}
The linear form $\ell =  ax+by+cz+dw$ must satisfy $d\neq 0$, for otherwise $F\in \langle x, y, z\rangle$ and again $F$ would have rank at most 3.
Finally, the change of coordinates 
$$
[x:y:z:w] \mapsto [\lambda x : y : z :  \lambda^{-1/q}w]
$$
where $\lambda^{q^2-1} = \frac{1}{d^{q}}$ 
transforms $F$ (formula~(\ref{d}))  into
$$
 (\lambda x)^q(a\lambda x+by+cz+d\lambda^{-1/q}w) + xw^q +  zy^{q}+yz^{q}
 $$
 which has the desired form since the coefficient of $x^qw$ is $ d \lambda^{q-\frac{1}{q}} =1$. Lemma~\ref{NiceForm} is proved.
\end{proof}

\begin{proof}[Proof of Proposition~\ref{transStars}] It suffices to show that 
 given an  arbitrary star point  $p$ on an arbitrary smooth extremal surface $X$ in $\mathbb P^3$ of degree $q+1$,  there is a choice of coordinates for $\mathbb P^3$ so that  $p = [0:0:0:1]$ and the defining form of $X$ is $w^qx + x^qw +y^qz + z^qy$.

Let $H$ be the  star plane  centered at $p$. Use  Lemma~\ref{NiceForm} to assume that  $p=[0:0:0:1]$, $H$ is defined by $x=0$ and that the Frobenius form defining $X$ looks like
$$
F = x^q(ax+ by+cz+w) + w^qx + y^qz+z^qy.
$$
We will perform a sequence of changes of coordinates that all fix $p$ and its tangent plane $H$, but eventually bring $F$ into the desired anti-diagonal form.

First,   we show we can change coordinates so as to assume $b=0$. Consider an indeterminate scalar $\lambda$.  Perform the change of coordinates
 $$
 [x: y : z : w] 
 \overset{\phi}\mapsto
  [x: y : z + \lambda^q  x : w - \lambda y].
  $$
The map $\phi$ fixes $p$ and $H$   but $\phi^*$ transforms the form $F$ into
$$
\begin{aligned}
& \ x^q(ax+by+cz + \lambda^q  cx+w - \lambda y) + (w- \lambda y)^qx+ y^q(z+\lambda^qx) +(z^q+\lambda^{q^2}x^q)y\\
 & \ = \ x^q((a+c\lambda^q) x + (b+\lambda^{q^2} - \lambda)y + c z + w)  + w^qx + y^qz + z^qy.
 \end{aligned}
$$
So any choice of $\lambda$ such  that $\lambda^{q^2} - \lambda+b =0$ will transform  $F$ into
$$
F_1 = 
 x^q(a' x  + c z + w) + w^qx+ y^qz + z^qy,
$$
 where $a'\in k$, without moving $p$ or $H$.  Similarly, interchanging the roles of $y$ and $z$, we can  transform $F_1$ into  
$$
F_2=  x^q\left(a'x + w \right)  + w^qx + y^qz + z^qy
$$ 
without moving star point $p$ or its star plane $H$.

Finally, again let $\lambda$ be an indeterminate scalar and consider the change of coordinates 
\begin{equation}\label{Tim}
 [x: y : z : w] 
 \overset{\phi}\mapsto
  [x: y : z  : w + \lambda x].
\end{equation}
This fixes $p$ and $H$ but transforms $F_2$ to 
$$(a'+ \lambda  + \lambda^q)x^{q+1}+ x^qw + w^qx +  y^qz + z^qy.$$
So choosing $\lambda$ to be any root of the polynomial $t^{q} + t + a'$, the form  $F_2$ is transformed into the standard form 
$
  x^qw  + w^qx + y^qz + z^qx
$
without changing the star point $p=[0:0:0:1]$.
This completes the proof of Proposition~\ref{transStars}.
\end{proof}



\subsection{Star Points and Star Planes}  We now  count configurations of star points on star planes and  on lines on the extremal surface:

\begin{theorem}\label{starcount}
Let  $L$ be an arbitrary line on a smooth extremal surface $X$. Then
\begin{enumerate}
\item[(a)] There are exactly  $q^2+1$ star points on $L$. Equivalently, there are exactly $q^2+1$ stars on the surface $X$ containing $L$. 
\item[(b)] There are exactly $q(q^2+1)$  lines on $X$ that intersect $L$, not counting $L$ itself.
\item[(c)] There are exactly $q^4$ lines on $X$ skew to $L$.
\end{enumerate}
\end{theorem}

Before proving Theorem~\ref{starcount}, we deduce a few corollaries.

\begin{corollary}\label{StarPointsOnStarPlane}
 Each star plane of  an  extremal surface  contains exactly $q^3+q^2+1$ star points---that is, each star  contains $q^3+q^2$ star points  other than its center.
\end{corollary}

\begin{proof}[Proof of  Corollary~\ref{StarPointsOnStarPlane}]
Let $p$ be the center  of the star $H\cap X$. Each of the $q+1$ lines in this contains exactly $q^2$ star points other than $p$ by Theorem~\ref{starcount}. So $H$ contains exactly $q^2(q+1)+1$ star points. \end{proof}

\begin{corollary} \label{count} Let $X$ be a smooth extremal surface of degree $q+1$.
\begin{enumerate}
\item[(a)] 
There are a total of $q^4+q^3+q+1 = (q^3+1)(q+1) $   distinct lines on $X$, each containing exactly $q^2+1$ star points.
\item[(b)]  There are a total of $q^5+q^3+q^2+1 = (q^3+1)(q^2+1)  $ distinct stars  on  $X$,  each containing exactly $q+1$ lines. 

\end{enumerate}
\end{corollary}

\begin{proof}
(a). Fix one line  $L$ on $X$. There are exactly $q(q^2+1)$ lines on $X$ which intersect $L$ by Theorem~\ref{starcount}(b). On the other hand, there are $q^4$ lines on $X$ disjoint from $L$ by Theorem~\ref{starcount}(c). So the total number of lines, counting $L$, is $q^4+q^3+q+1$.

(b). There are a total of $q^4+q^3+q+1$ lines, and each line is contained in exactly $q^2+1$ stars. So the number of pairs $(L, H)$ consisting of a line $L$ on a star $H\cap X$ must be $(q^4+q^3+q+1)(q^2+1)$. On the other hand, each 
 star contains exactly $q+1$ lines,  so the total number of stars is
 $$
\frac{(q^4+q^3+q+1)(q^2+1)}{q+1} =  \frac{(q^3+1)(q+1)(q^2+1)}{q+1} =(q^3+1)(q^2+1) = q^5+q^3+q^2+1.
$$

\end{proof}

\begin{proof}[Proof of Theorem~\ref{starcount}] 
(a).
The line $L$ belongs to some star $H\cap X$  by Corollary~\ref{LineInStar}.   By Corollary~\ref{transLines}, we can 
choose coordinates so that $X$ is defined by 
$$
F = x^qw + xw^q +  y^{q}z+z^{q}y,
$$
and $L\subset H$ are  cut out by $x, y$ and $x$ respectively.

Consider the pencil of planes containing the line $L$. Each  plane  $H_{\lambda}$ in the pencil  is defined by the vanishing of some linear form  $\lambda x - y$. The plane $H$ itself is defined by $x=0$ (the case where $\lambda=\infty$), which we already assumed is  a star.

Restricting the Frobenius form $F$ to the plane $H_{\lambda}$, we can set $y=\lambda x $ and view the  plane section $X\cap H_{\lambda}$ as defined by the Frobenius form
$$
\overline{F}  = x^qw + xw^q + \lambda^q x^q z +  \lambda x z^q,
$$
 in the variables $x, z, w$.  
 The plane section $X\cap H_{\lambda}$ is a star if and only if the form $\overline F$ is degenerate (Cf. Proposition~\ref{planeSection}). 
 
 We claim that  $\overline{F} $ is degenerate precisely when $\lambda$ is a root of the separable polynomial
 $
t^{q^2} - t. 
 $
This will imply that there are precisely $q^2$  planes (besides $H$) which contain $L$ as a component of a star, so the proof of (a) will be complete once we have proved the claim.

 To this end,  consider the  change of coordinates
 $$
 [x:z:w] \mapsto [x:z:w-\lambda^qz]. 
 $$
 This transformation sends $\overline{F}$ to
 $$
\overline{F_1} = x^qw+xw^q + (\lambda-\lambda^{q^2}) x z^q,
 $$
 which is clearly degenerate if 
$ \lambda^{q^2}- \lambda = 0.$
On the other hand, if $\lambda^{q^2}- \lambda \neq  0$, then $\overline{F}$  is {\it not degenerate}. Indeed,  in this case
$$
  \overline{F_1}= x((\gamma z  + w)^q  + x^{q-1}w) $$
 for some non-zero $\gamma$, which  is projectively equivalent to $ x(z^q + x^{q-1}w)$ , so  defines
a union of  the line $L$ and an irreducible curve 
of degree $q$, not a star. This completes the proof of (a).

(b).  A  line $M$ on $X$ intersects $L$ if and only if $L$ and $M$ appear together  in a star.
There are $q^2+1$ stars containing $L$ and each of them contains $q$ distinct lines (other than $L$). Of course, a pair $L$ and $M$ can not appear together in more than one star, since the plane producing a  star is uniquely determined by any two lines in it. So there must be 
$q(q^2+1)$ distinct lines $M$ which intersect $L$ on our extremal surface.

(c).
Fix a star  $H\cap X$ containing $L$ (this is possible by Corollary~\ref{LineInStar}). 
There are $q$ other lines in this star. Pick one, $M$. Now $M$ appears in exactly $q^2$ other stars besides $H$ by Theorem~\ref{starcount}(a). For each of these stars, each of the other $q$ lines in the star is a  line $L'$ which does not meet $L$. Indeed,  if $L'$ meets $L$, then the lines $L, L', M$ form a triangle, contradicting Corollary~\ref{noTriangle}. In this way, we  produce $q^3$ distinct lines $L'$ on $X$ which meet $M$ but not $L$. Now,  varying over each of the $q$ lines $M$ in the star $H\cap X$ (other than $L$), we produce $q^3$ new lines for each of the $q$ choices of line $M$. 
In total, we found $q^4$ lines skew to $L$.

Finally, we need to show that our above count includes {\it every} line $L'$ on $X$  skew to $L$.  Say $L'$ is skew to $L$. Pick any star point  $p$ on $L$. The star plane $H$ at $p$ contains $L$ but not $L'$ (otherwise $L'$ would meet $L$). So $L'$ must meet $H$ at some point $p'$, necessarily in the star $H\cap X$. So $p'$ is a star point on some line $M$ in  the star $X\cap H$, and the star centered at $p'$ contains $L'$. This means $L'$ is a line of the type we already counted in the previous  paragraph.
So there are exactly $q^4$ lines on the extremal surface  skew to any fixed line on the surface. 
\end{proof}


\begin{corollary}\label{skewpair-connected}
Fix any pair of skew lines on an extremal surface. Then there are {\it exactly}  $q^2+1$ lines on the surface that meet both.\end{corollary}

\begin{proof}[Proof of Corollary~\ref{skewpair-connected}] 
Fix arbitrary skew lines $L$ and $L'$ on the extremal surface $X$. We claim that for each star point $p$ on $L$, there is exactly one line  through $p$ meeting $L'$. Because there are exactly $q^2+1$ star points on $L$ (Theorem~\ref{starcount}(a)) and any intersection point of lines on $X$ is a star point, the claim proves the corollary.

To prove the claim, observe that $L'$ is not in the star plane $H$ centered at $p$, since that would imply $L'$ meets $L$. Thus $L'$ meets $H$ at a unique point $p'$, which means $p'$ is in the star $X\cap H$, and hence in (exactly) one of the lines $M$ in the star $X\cap H$. The line $M$ meets both $L$ and $L'$. There is no other line through $p$ meeting both $L$ and $L'$, for if $M'$ is another, then $M, M', L'$ form a triangle, contrary to Corollary~\ref{noTriangle}.
\end{proof}


\subsection{The Automorphism Group of an Extremal Surface}\label{automorphisms}
We now use the geometry of extremal surfaces  to describe their automorphism groups.  The results in this section are (essentially) known, albeit in somewhat different contexts with slightly stronger hypotheses; we 
include straightforward new proofs for completeness. The first result is due to Shioda when $d>3$ 
{\cite[p97]{Shioda} and Duncan and Dolgachev when $d=3$ {\cite[\S~5.1]{dolgachev-duncan.automorphisms}.

\begin{theorem} \label{autogroup}  Let $X$ be a smooth extremal hypersurface of  degree $q+1$ and dimension $n-2\geq 0$ over algebraically closed field $k$.
The group  $\Aut(X)$ of projective linear automorphisms of  $X$ is isomorphic to the finite group $\PU(n, \mathbb F_{q^2})$, where 
 $\PU(n, \mathbb F_{q^2})$ is the quotient of the finite unitary group
$$
\U(n, \mathbb F_{q^2}) = \left\{g \in \GL(n, \mathbb F_{q^2})\;\; \big| \;\; (g^{[q]})^\top \, g = I_n\right\}  
$$
by its center,  $$\{ \lambda I_n\;\; | \;\; \lambda^{q+1} = 1\}, $$ the cyclic group of scalar matrices of order $q+1$.
\end{theorem}

\begin{remark}\label{Dim0}
We have already computed that when $X$ is zero dimensional, $\Aut(X)$ is isomorphic to $\PGL(2, \mathbb F_q)$ (Proposition~\ref{dim0}), so Theorem~\ref{autogroup} confirms that 
$\PU(2, \mathbb F_{q^2}) \cong \PGL(2, \mathbb F_q)$.
\end{remark}

\begin{proof}[Proof of Theorem~\ref{autogroup}]
Choose coordinates  so that the extremal hypersurface $X$ is defined by
 $F= x_1^{q+1} + x_2^{q+1} + \dots + x_n^{q+1}$ (Theorem~\ref{smooth}).  The group $\Aut(X)$ is the subgroup of $\PGL(n, k)$  represented by matrices $g\in \GL_n(k)$ such that $g^* F  = \lambda F$ for some non-zero scalar $\lambda.$ Because $k$ is algebraically closed, the class of $g$ in $\PGL(n, k)$  can be represented by
   the scalar multiple $\mu g$ where $\mu^{q+1} = \frac{1}{\lambda}$, so without loss of generality we assume that  $g^* F = F$.
Such $g$ satisfy
\begin{equation}\label{g}
 (g^{[q]})^\top \, I_n \  g =  I_n, 
\end{equation}
by formula~(2) in Section~2. 
 Raising (\ref{g})  to the $q^{th}$-power and transposing, we have 
$
 (g^{[q]})^\top g^{[q^2]}  = I_n.
$\
 In particular, both $g$ and $g^{[q^2]}$ are inverses of the matrix  $(g^{[q]})^\top$, so that 
 $
g = g^{[q^2]}.
$
 Thus each entry of $g$ is  fixed by the Frobenius map $x\mapsto x^{q^2}$, and hence in $\mathbb F_{q^2}$---that is, we can assume $g\in \PGL(4, \mathbb F_{q^2})$.
 This means that the naturally induced group map 
  $$\U(n, \mathbb F_{q^2}) \overset{\gamma}\longrightarrow {\Aut} (X)$$ 
is surjective, and it  remains only to  compute its kernel.

  An element  $g\in \U(n, \mathbb F_{q^2})$ induces the identity map on $X$ if and only if $g=\lambda I_n$ for some $\lambda \in k^*$.
 But a scalar matrix $\lambda I_n$ is in $\U(n)$ if and only if $(\lambda I_n)^{[q]}(\lambda I_n) = \lambda^{q+1} g^{[q]} g^\top =  I_n$---that is, if and only if $\lambda^{q+1} = 1$. 
 On the other hand, all $(q+1)$-st roots  $\lambda$ of unity in $k$ are  in $\mathbb F_{q^2}$:  if $\lambda^{q+1}=1$, then
  $\lambda^{q^2} = \lambda^{q^2-1} \lambda = (\lambda^{q+1})^{q-1} \lambda = \lambda.$
So  the kernel is  the cyclic group of order $q+1$ consisting of the scalar matrices of $q+1$-st roots of unity, as claimed. The theorem is proved.
\end{proof}

\begin{remark}\label{CompareShioda} Shioda's theorem is actually concerned with {\bf Fermat hypersurfaces}---projective  hypersurfaces defined by 
 $x_0^d+x_1^d+\cdots+ x_n^d$ over an arbitrary algebraically closed field (although he omits the case where $d=3$, which our result includes). Shioda shows that, \emph{with the  exception of the  case  where $d=p^e+1$}, the automorphism group of the Fermat hypersurface is generated by the "obvious" automorphisms: the  $S_n$ permuting the coordinates and the group
  $(\mu_{d}^n/\mu_{d}) \cong \mu_{d}^{n-1}$ scaling the coordinates by $d$-th roots of unity. When $d=p^e+1>3$, he proves Theorem~\ref{autogroup} for the Fermat hypersurface using a different method than our argument. Likewise, Duncan and Dolgachev are concerned with the automorphism group of cubic surfaces in general; in the special case of the Fermat cubic surface of characteristic two (the extremal case), their proof uses a result of Beauville to show there is no "canonical point", and then appeals to known facts about the automorphism group of a finite Hermitian geometry \cite[19.1.7, 19.1.9] {Hirschfeld}.  \end{remark}

\begin{remark}\label{otherForms}  In the special case that extremal surface defined by a {\it Hermitian} form (meaning that its matrix $A$ satisfies $A^{[q]}=A^\top$), 
 the  proof of Theorem~\ref{autogroup} shows that
$$
{\Aut}(X) =  \{g \in \GL(n, \mathbb F_{q^2}) \; | \; (g^{[q]})^\top A\, g  = A\}/\{\lambda I_n\; | \; \lambda^{q+1}=1\}.$$
For an arbitrary smooth extremal surface $X$, the automorphism group  $\Aut(X)$ is  conjugate to $\PU(n, \mathbb F_{q^2})$ via the automorphism in $\PGL(n, k)$ taking it to the Fermat hypersurface.  
\end{remark}


\subsection{\texorpdfstring{Star Points and $\mathbb F_{q^2}$-rational points}{Star Points and Fq2-rational points}}\label{StarOnHerm}
We next give a straightforward proof of a familiar fact that connects extremal surfaces to finite geometry:

\begin{proposition}\label{HermStar} The star points on a  extremal surface  of degree $q+1$ defined by a Hermitian form are precisely its $\mathbb F_{q^2}$ points.
\end{proposition}

\begin{remark}  Proposition~\ref{HermStar} (together with Corollary~\ref{count}) recovers the known fact that there are precisely $q^5+q^3+q^2+1$ \,\, points in a Hermitian sub-geometry of the finite projective 3-space over $\mathbb F_{q^2}$
\cite{Segre.67}, \cite[19.1.5]{Hirschfeld}. 
\end{remark}

\begin{proof}
We first verify that the star points of  extremal surface $X$ defined by  $x^qw+w^qx + y^qz+z^qy$ are precisely its $\mathbb F_{q^2}$ points. 
 By symmetry, 
it suffices show this for   the open set $\mathcal U$ where $x \neq 0$.

Consider  a point $p=[1:a:b:c] \in \mathcal  U$ whose coordinates are in  $\mathbb F_{q^2}$.  The tangent plane $T_pX$ at $p$  is\,
$
c^qx + b^qy + a^q z + w,
$\,  so  the plane section  $T_pX\cap X$ is defined by  the Frobenius form in $x, y, z$
\begin{equation}\label{intersection}
-x^q(c^qx + b^qy + a^q z)  - x(c^qx + b^qy + a^q z)^q + y^qz + z^qy.
\end{equation}
Thus $p$ is a star point  if and only if (\ref{intersection})  is a degenerate  Frobenius form (Proposition~\ref{planeSection}).
Because the projective transformation
\[
\phi: [x:y:z] \mapsto [x:y+ax: z+bx], 
\]
 transforms the form  (\ref{intersection}) into the degenerate form $y^qz+z^qy$ (remember that $c + c^q + a^qb + b^qa = 0$), we conclude that all $\mathbb F_{q^2}$-points of $\mathcal U$ are star points.

To check that the $\mathbb F_{q^2}$-points comprise  {\it all} star points of $\mathcal U$, we count them.  Note{\footnote{Proof: 
The map $\mathbb F_{q^2} \overset{\gamma}\longrightarrow \mathbb F_{q^2}$ sending $t\mapsto t^q+t$ is linear over the subfield $\mathbb F_{q}$, and its kernel consists of $q-1$ distinct $(q-1)$-st roots of $-1$ together with $0$. So there are exactly $q$ solutions  to $\gamma(t) = \gamma(-a^qb) = -a^qb-ab^q$ in $\mathbb F_{q^2}^2$ as well.
}}
 that for each choice of the pair $(a, b)\in \mathbb F_{q^2}^2$,  the polynomial $t^q+t + (a^{q}b + b^{q}a)$ has $q$ distinct solutions in $\mathbb F_{q^2}$.
  So there are exactly  $q^5$ \,    $\mathbb F_{q^2}$-rational points  $[1:a:b:c]$ \ in $\mathcal U$, all of which are star points. On the other hand, there are exactly $q^5$ star points in $\mathcal U$ as well---of the total $q^5+q^3+q^2+1$ star points on $X$ (Theorem~\ref{count}(b)), there are precisely $q^3+q^2+1$ in the complement of  $\mathcal U$, because $X\setminus \mathcal  U=\mathbb V(w)$   is a  star plane (Corollary~\ref{StarPointsOnStarPlane}). By the pigeon hole principle, the  star points of $\mathcal U$ (and hence of $X$)   are precisely its $\mathbb F_{q^2}$-points.

Now consider an arbitrary  smooth variety $X'$, given by some 
Frobenius form that is Hermitian---that is, whose representing matrix $A$ satisfies $(A^{[q]})^\top = A$. There is a change of coordinates $g\in PGL(4, k)$  transforming $X'$ to  the extremal surface $X = \mathbb V(x^qw+w^qx + y^qz+z^qy)$, so that using formula~(\ref{action}) in \S~\ref{coordinateChange}, we have
\begin{equation}\label{J}
(g^{[q]})^\top J g = A
\end{equation}
where 
$J = 
\left[
\begin{smallmatrix} 0 &0 & 0 &1\\
0 &0 & 1 &0\\
0 &1 & 0 &0\\
1 &0 & 0 &0\\
\end{smallmatrix}
\right]
$ is the matrix of the Frobenius form defining  $X$.
Raising (\ref{J})  to the $q$-th power and transposing, we have 
\begin{equation}\label{JJ}
(g^{[q]})^\top (J^{[q]})^\top g^{[q^2]} = (A^{[q]})^\top = A.
\end{equation}
Setting (\ref{J}) and (\ref{JJ}) equal and remembering that $(J^{[q]})^\top =J$ , we have 
\[(g^{[q]})^\top J g^{[q^2]}  = (g^{[q]})^\top J g,\]
from whence it follows  that $g^{[q^2]} = g$. This tells us that  $g\in \PGL(4, \mathbb F_{q^2})$.

Now because $g^{-1}$ is an isomorphism from  $X$ to $X'$, it defines a bijection between their respective  star points. Because $g$ has entries in $\mathbb F_{q^2}$ it also preserves the $\mathbb F_{q^2}$-rationality of points. So the star  points of $X'$ are precisely its $\mathbb F_{q^2}$-points. 
\end{proof}


\section{star chords}

Star chords  are auxiliary lines  not  on the extremal surface but none-the-less intimately related:
\begin{definition}
A {\bf star chord} for  a smooth extremal surface $X$ is a line in $\mathbb P^3$ {\it not on $X$} which passes through (at least) two star points of $X$.
 \end{definition}

\begin{remark}\label{caution} Despite the name,  a star chord $\ell$  through  star point $p$ is never in the star plane $T_pX$ centered at $p$. Otherwise, assume $\ell\subset T_pX$. Because there is another star point  $p'\in \ell$, the point $p'$ would then be on some line $L$ in the star $T_pX\cap X.$ But then  both $L$ and $\ell$ contain both $p$ and $p'$, which means $\ell = L$, contrary to the fact that $\ell\not\subset X$.
\end{remark}

\begin{remark} In the special case where the extremal surface is defined by a Hermitian form over $\mathbb F_{q^2}$,  star chords are  {\it Baer sublines} or {\it hyperbolic lines} in the terminology of finite geometry (see {\it e.g.} \cite[p4]{Bamberg+Durante} or \cite[p102]{Masini}). In this context, lines on the  surface are called its {\it generators}. 
\end{remark}

The basic facts about star chords are the following:

\begin{theorem}\label{StarLines}  Let $\ell$ be an arbitrary  star chord for a smooth  extremal surface. Then
\begin{enumerate}
\item[(i)]  The  stars centered at points on $\ell$ share no lines.
\item[(ii)]  The star planes of  all stars centered along $\ell$  intersect in a common line $\ell'$ which is skew to $\ell$ and  also a star chord for $X$.
  \item[(iii)]  The star planes of all stars centered along $\ell'$ intersect in the original star chord $\ell$.
\item[(iv)]
The star chords $\ell$ and $\ell'$ each  intersect $X$ in $q+1$ distinct star points.
\end{enumerate}
\end{theorem}

Before proving Theorem~\ref{StarLines}, we observe that it ensures  that  the next definition makes sense.

\begin{definition}\label{dual}
The {\bf dual of a star chord} $\ell$ for an extremal surface  is the unique  star chord $\ell'$ contained in all star planes  centered along $\ell$, or equivalently, the intersection of all star planes centered along $\ell$.
\end{definition}

Duality between star chords is a symmetric relationship: Theorem~\ref{StarLines}(iii) implies that $\ell'$ is the dual star chord of $\ell$ if and only if $\ell$ is the dual star chord of $\ell'$.

\begin{example}\label{standardStarLineConfig} 
The lines  $\ell=\mathbb V(x, y)$  and $\ell'=\mathbb V(z, w)$ are a pair of dual star chords on the 
 Fermat extremal surface $X = \mathbb V(x^{q+1}+y^{q+1}+z^{q+1}+w^{q+1})$. Indeed, $\ell$ is not on $X$ but contains the  $q+1$ star  points $p_a  = [0:0:a:1]$, where $a^{q+1}=-1$.
  To check that $p_a$ is a star point, observe that the tangent plane to $p_a$ is  $T_{p_a}X = \mathbb V(a^qz+w) = \mathbb V(z-aw),$ which intersects $X$ in a star. 
 These star planes $\mathbb V(z-aw)$ all obviously contain $\ell'$, so $\ell'$ is their common intersection,  as promised by Theorem~\ref{StarLines}. Note that, dually, the star points on  $\ell'$ are the points
 $p_b'=[b:1:0:0]$ where $b^{q+1}=-1,$ and the corresponding star planes $\mathbb V(x-by)$ intersect in $\ell$.
\end{example}


\begin{proof}[Proof of Theorem~\ref{StarLines}]

Fix any two star points $p_1$ and $p_2$ on $\ell$. Since $p_1$ is on {\it every} line in the star centered at  $p_1$, and likewise for $p_2$, any shared line 
  shared line between these stars would contain both $p_1$ and $p_2$ and hence be $\ell$ itself. But by definition, the star chord $\ell$ is not on $X$. So stars centered on $\ell$ can  not share any lines, proving (i).

Now,  let $\ell' = T_{p_1}X \cap T_{p_2}X$.  
Note that $\ell' \not\subset X$: otherwise, $\ell'\subset T_{p_1}X\cap X$ and $\ell'\subset T_{p_2}X\cap X$, making $\ell'$ a shared line between these stars, which would contradict (i). 

We claim that  $\ell'$ is skew to $\ell$.  First note that $\ell\neq \ell'$, for otherwise the star chord $\ell$ lies in the star plane $T_{p_1}X$, contradicting Remark~\ref{caution}. So at least one of $p_1$ or $p_2$---say $p_1$---is not on $\ell'$.
 Now, if  $\ell$ and $\ell'$ are not skew, the  unique plane they span  is  necessarily the plane $T_{p_1}X$, since both   planes contain $\ell' $ and $p_1\not\in \ell'$.
 But now the star chord $\ell$ is in the star plane  $T_{p_1}X$, again contradicting Remark~\ref{caution}.

We now claim $\ell'$ is a star chord intersecting $X$ in $q+1$ distinct star points.
Observe that 
because $\ell' \subset T_{p_1}X$, it meets each line in the star $T_{p_1}X\cap X$. But since the center $p_1$ is not on $ \ell'$, we know $\ell'$ must meet each of the $q+1$ lines in the star $T_{p_1}X\cap X$ in a {\it distinct} point. These $q+1$ points make up the full intersection 
 $\ell'\cap X$, since $X$ has degree   $q+1$. Similarly, since also $p_2\not\in\ell'$, the points of  $\ell'\cap X$ are the  $q+1$ distinct intersection points of $\ell'$ with the lines in the star 
  $T_{p_2}X\cap X$.  Thus   each $p'$ in $\ell'\cap X$ lies on  at least two lines of $X$.
  So $\ell'$ is a star chord and meets $X$ in $q+1$ distinct star points.

Next, we show that
$\ell\subset T_{p'}X$ for {all} \, $p'\in \ell'\cap X$, which will establish (iii). As we saw in the preceding paragraph, the star $X\cap T_{p'}X $ contains a line in each of the two stars $T_{p_1}X\cap X$ and  $T_{p_2}X\cap X$. In particular, both $p_1$ and $p_2$ are   in  $T_{p'}X$, so also
$\ell=\overline{p_1p_2}\subset T_{p'}X$.
  
 We now claim $\ell $ meets $X$ in $q+1$ distinct star points. To see this, 
 take an arbitrary  $p'\in \ell'\cap X$. Since $\ell \subset T_{p'}X$  (using (iii)) but $p'\not\in\ell$ (by skewness of $\ell$ and $ \ell'$),  each line in the star $T_{p'}X\cap X$ meets $\ell$ in a distinct point.  These are the $q+1$ points of $X\cap \ell$.
  They are star points because each lies on a line in every other star $T_{p''}X\cap X$  with $p''\in \ell'$.

   The proof will be complete once we have shown that  $\ell'$ is independent of the choice of the star points $p_1$ and $p_2$ on $\ell$.
For this, it suffices to show that  $\ell'\subset T_pX$ for each star point  $p$ on $ \ell$, so  that $\ell' $ is the intersection of all $q+1$ star planes centered along $\ell$ (or any two of them).
But taking any $p'\in \ell'\cap X$, we have seen that $p'$ lies on a line in the star $T_pX\cap X$.
  So the $q+1$ star points of $\ell'$, and hence $\ell'$ itself, are in  $T_pX$.
\end{proof}

 
 \subsection{Symmetry of star chords}

\begin{theorem}\label{StarLineTrans}
The automorphism group of a smooth extremal surface induces a natural {\it transitive}  action  on the set of all its star chords.
\end{theorem}

\begin{proof}[Proof of Theorem~\ref{StarLineTrans}]
A star chord is determined by two star points not spanning a line on $X$, so any projective linear automorphism of  the surface induces a permutation of the star chords for the surface. 
To prove this action is transitive, it suffices to show that $\Aut(X)$ acts transitively on the set of ordered pairs $(p_1, p_2)$ of star points spanning  star chords.

Fix an arbitrary ordered pair $(p_1, p_2)$ of star points spanning a star chord.
Since  $\Aut(X)$ acts transitively on  star points (Proposition~\ref{transStars}), there is no loss of generality in assuming 
 $$X=\mathbb V(x^qw+w^qx + y^qz+z^qy)  \quad \;\; {\text{ and }}  \quad \;\;  p_1=[0:0:0:1].$$
 The theorem will be proved if we show that, in addition, we can choose coordinates so that 
$
p_2$ is the star point $[1:0:0:0].$  

First note that we can assume that $p_2=[1:a:b:c].$  Indeed, otherwise $p_2\in \mathbb V(x) =  T_{p_1}X$, so that $p_2$ would be in the star centered at $p_1$. In this case, the line 
$\overline{p_1p_2}$ is in that star  and hence on $X$, contrary to the assumption that $p_1$ and $p_2$ span a  star chord. Note also that $a, b, c\in \mathbb F_{q^2}$ (Proposition~\ref{HermStar}).
 
Consider the change of coordinates $\phi$
$$
[x:y:z:w] \overset{\phi}{\mapsto} [x:y-ax: z-bx:  w +c^q x + b^ q y+a^q  z].
$$
Clearly $\phi$ fixes $[0:0:0:1]$ and takes $p=[1:a:b:c]$ to $[1:0:0:0]$. Furthermore, $\phi^*$  fixes the polynomial $x^qw+w^qx+y^qz+z^qy$: Remembering that 
$a^{q^2}=a$, $b^{q^2} = b$,  $c^{q^2}=c$, and $c^q+c=-(a^qb+b^qa)$, we easily verify that
$\phi^*(x^qw+w^qx + y^qz+z^qy)= x^qw+w^qx+y^qz+z^qy.
$
This completes the proof of Theorem~\ref{StarLineTrans}.\end{proof} 

Theorem~\ref{StarLineTrans} has  the following consequence:

\begin{corollary}\label{DoubleStarTransCor}
The automorphism group of a smooth extremal surface $X$ acts transitively on the set of pairs of skew lines on $X$.
\end{corollary}

\begin{proof}[Proof of Corollary~\ref{DoubleStarTransCor}]
Fix an arbitrary pair of skew lines $L$ and $L'$ on an extremal surface $X$. 
It suffices to show that we can choose projective coordinates so that $X$ is defined by the Frobenius form $F=x^{q}w+ w^qx+y^{q+1}+z^{q+1}$ and the two lines  are
$L_a=\mathbb V(x, y-a z)$ and $L_b=\mathbb V(w, y-bz)$  where $a$ and $b$ are distinct fixed ${q+1}$ roots of $-1$. 

We can choose star points  $p\in L$ and  $p'\in L'$    that span a star chord $\ell$.  Indeed, 
 there are $(q^2+1)^2$ lines in $\mathbb P^3$ connecting star points on $L$ to star points on $L'$  (Theorem~\ref{starcount}(a)) but only $q^2+1$ of them lie on $X$  (Corollary~\ref{skewpair-connected}). 

Now  by (the proof of) Theorem~\ref{StarLineTrans}, the automorphism group of an extremal surface acts transitively on ordered pairs of star points spanning star chords. 
So we can choose coordinates so that  the extremal surface is
 $X= \mathbb V(x^qw+w^qx+y^{q+1}+z^{q+1})$,  $p=[0:0:0:1]$, and $p'=[1:0:0:0]$ (note that $\overline{pp'}=\mathbb V(y, z)$, which is not on $X$).
  In this case, the star planes at $p$ and $p'$, respectively, are defined by $x$ and $w$. 
 The line $L$ is therefore  in the star $T_pX\cap X =\mathbb V (x, y^{q+1}+z^{q+1})$
 and $L'$ is  the star $T_{p'}X\cap X =\mathbb V (w, y^{q+1}+z^{q+1})$. In particular, 
  $L=\mathbb V(x, y-\nu_1 z)$ and $L'=\mathbb V(w, y-\nu_2 z)$ where $\nu_1^{q+1}= \nu_2^{q+1}=-1$. 
  The assumption that $L$ and $L'$ are skew means that  $\nu_1 \neq \nu_2$.

Finally, we need a change of coordinates that fixes $x$ and $w$, while taking the forms $\{y-\nu_1 z, y-\nu_2 z\}$ to our chosen ones $\{y-a z, y-b z\}.$
Equivalently, we need  the automorphism group of  the extremal configuration $Y=\mathbb V(y^{q+1}+ z^{q+1})$ in $\mathbb P^1$ to act two-transitively on $Y.$  This is immediate from Proposition~\ref{dim0}.\end{proof}

\subsection{Application: The order of the automorphism group} 
The group $\PU(n, \mathbb F_{q^2})$ is well-studied in representation theory, and its order is classically known \cite[pp131-144]{Dickson}, \cite[19.1.6]{Hirschfeld}. Still, we can give a cute computation of the order of  
$\PU(4, \mathbb F_{q^2})$  as a corollary of Theorem~\ref{StarLineTrans}:

\begin{corollary}\label{orderAut} The  order of the automorphism group  of a smooth extremal surface, and hence the order  of $\PU(4, \mathbb F_{q^2})$, is  $${q^6(q^2-1)(q^3+1)(q^4-1)}.$$
\end{corollary} 

\begin{proof}[Proof of Corollary~\ref{orderAut}] Fix a smooth extremal surface $X$.  
We compute the order of  $\Aut(X)$ using the orbit-stabilizer theorem for its action on the set $\mathcal S = \{(p, p') \; | \; p, p' {\text{ star points, }} \overline{p p'}\not\subset X\} $ of ordered pairs
 of star points spanning a star chord. This action is transitive by (the proof of) Theorem~\ref{StarLineTrans}. 

The cardinality of the orbit $\mathcal S$ is $q^5(q^3 + 1)(q^2+1)$. Indeed, there are $ (q^3 + 1)(q^2+1)$  star points on $X$  (Corollary~\ref{count}), 
so it suffices to show that for each  choice of star point $p$,  there are $q^5$ star points $p'$ such that $\overline{pp'}$ is a star chord. For this, note that  of the total $q^5+q^3+q^2$ star points  other than $p$, 
there  are  $q^3+q^2$  in the  star plane centered at $p$ (Corollary~\ref{StarPointsOnStarPlane}). This leaves $q^5$ star points $p'$ such that $\overline{p p'}$ is a star chord, establishing that the cardinality of $\mathcal S$ is as claimed. 

  Corollary~\ref{orderAut} then follows immediately from the following computation:
\begin{lemma}
The stabilizer of an ordered pair of star points $(p_1, p_2)$ spanning a star chord  on a smooth  extremal surface is isomorphic to 
\begin{equation}\label{stabilizer}
 \mathbb F_{q^2}^* \times \U(2, \mathbb F_{q^2}) \, / \, \{(\lambda, \lambda I_2)\;\; | \;\; \lambda^{q+1}=1\},
\end{equation}
and has order ${q(q^2-1)^2}$.
\end{lemma}

\begin{proof}[Proof of Lemma] We may assume that 
$
X=\mathbb V(x^qw+ y^{q}z+z^{q}y + w^qx), 
 p_1=[1:0:0:0],  $ and $ p_2=[0:0:0:1].$
 Then 
any automorphism in the stabilizer of  $(p_1, p_2)$ stabilizes also the tangent planes at $p_1$ and $p_2$ so is
represented by a matrix of the form 
$$
\begin{bmatrix} a_{11} & 0 & 0 & 0 \\
0 & a_{22} & a_{23} & 0 \\
0 & a_{32} & a_{33} & 0 \\
0 & 0 & 0& a_{44} \\
\end{bmatrix}
$$ fixing $x^qw+ y^{q}z+z^{q}y + w^qx.$ In particular, 
 the  submatrix 
$h = \begin{bmatrix}
 a_{22} & a_{23} \\
 a_{32} & a_{33}  \\
\end{bmatrix} $ 
fixes $y^{q}z+z^{q}y$, so is in $ \U(2, \mathbb F_{q^2})$.
Furthermore,  because $g$ fixes  $x^qw+w^qx$,  we get 
 $a_{11}^{q^2} = a_{11}$ and $ a_{44}^{q^2}=a_{44} = \frac{1}{a_{11}^q}$. 
Thus there is a natural surjective map
$$
  \mathbb F_{q^2}^* \times \U(2, \mathbb F_{q^2})  \rightarrow  {\text{Stab}} \{p, p'\} \qquad (a_{11},  h) \mapsto [g] =  
  \begin{bmatrix} a_{11} & 0 & 0 & 0 \\
0 & a_{22} & a_{23} & 0 \\
0 & a_{32} & a_{33} & 0 \\
0 & 0 & 0& a_{11}^{-q} \\
\end{bmatrix},
$$
and we
easily compute that the kernel is  $\{(\lambda, \lambda I_2)\;\; | \;\; \lambda^{q+1}=1\}$. 
So (invoking  Remark~\ref{Dim0}), 
 the order of the stabilizer of $(p_1, p_2)$ is
\[
\frac{|\mathbb F^*_{q^2}|\ |  \U(2, \mathbb F_{q^2})| } {|\mu_{q+1}|}  =
|\mathbb F^*_{q^2}| |\PU(2, \mathbb F_{q^2})| =   (q^2-1) |\PGL(2, \mathbb F_q)|  = q(q^2-1)^2.
\]

\end{proof}
\end{proof}



 \section{Quadric Configurations}

Extremal surfaces contain interesting line configurations we call {\bf quadric configurations:}

\noindent
\begin{minipage}{0.7\textwidth}
\begin{definition} A \textbf{quadric configuration} on a surface  of degree $d\geq 3$  in projective three space  is a collection of $2d$ lines on the surface  consisting of two sets of $d$ skew lines with the property that each line in either set meets every line of the other set.  
\end{definition}
\end{minipage}
\hfill
\begin{minipage}{0.3\textwidth}
\vskip .4cm
\hskip 1cm \includegraphics{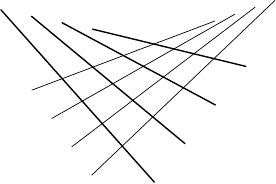}

\captionof*{figure}{\tiny{$d=4$}}
\label{fig:deg4QuadConfig}

\end{minipage}

The next proposition justifies the name:

\begin{proposition}  A quadric configuration on an irreducible  surface $X$   is equal to $X\cap Q$ for some unique smooth quadric surface $Q$.
\end{proposition}

\begin{proof}
Let  $\mathcal L \cup \mathcal M$ be a configuration of lines, where  $\mathcal L$ (respectively $\mathcal M$) consists of $d$ skew lines   intersecting every line in $\mathcal M$ (respectively $\mathcal L$).     Choose any three skew lines  $L_1, L_2, L_3\in \mathcal L$, and let  $Q$ be the unique smooth quadric they determine \cite[2.12]{Harris}. 
The lines of $\mathcal M$ intersect {\it all  } lines in $\mathcal L$, including  $L_1$, $L_2$, and $L_3$, which lie on $Q$. So each line  $M\in \mathcal M$ intersects the quadric $Q$ in at least three points, which means $M\subset Q$. But now each line $L\in \mathcal L$ intersects all  lines in $\mathcal M$, so $L$ intersects $Q$ in at least three points. Again, we conclude $L\subset Q$.  So $\mathcal L \cup \mathcal M\subset Q$. 

Now if $\mathcal L \cup \mathcal M \subset X$, then $\mathcal L\cup \mathcal M \subset X\cap Q$. So since $X\cap Q$ and $\mathcal L \cup \mathcal M $ both have degree $2d$, and $X\cap Q$ is a complete intersection, we conclude that $X\cap Q$ is precisely the reduced union of the $2d$ lines in 
 $\mathcal L\cup \mathcal M$.\end{proof}

\begin{example} \label{exQuadric2} Let $Q_{\mu}$ be   the quadric surface  $Q_{\mu}=\mathbb V(\mu xw-yz)$, where
$\mu\in k$ is a fixed $(q+1)$-st root of unity. The quadric  $Q_{\mu}$ defines a quadric configuration on the Fermat extremal surface. 
 Indeed, 
the lines in the sets  
$$
\begin{aligned}
\mathcal L_{\mu} &= \{\mathbb V(x-\alpha y,  z-\mu \alpha w) \;\; | \;\; \alpha^{q+1} = -1\}\\
 \mathcal M_{\mu} &=  \{\mathbb V(x-\beta z,  y-\mu \beta  w) \;\; | \;\; \beta^{q+1} = -1 \}
 \end{aligned}
$$
 all lie on the quadric $Q_{\mu}$ (with the lines in $\mathcal L_{\mu}$ and $\mathcal M_{\mu}$ in opposite rulings),   as well as on the extremal surface $X=\mathbb V(x^{q+1}+y^{q+1}+z^{q+1}+w^{q+1})$. Thus 
 $X\cap Q_{\mu}$ is the quadric configuration $\mathcal L_{\mu} \cup
 \mathcal M_{\mu}.$
 \end{example}

Quadric configurations are rare on an arbitrary surface---for example, a generic  surface of degree greater than three  admits {\it no lines at all} \cite[12.8]{Harris}.
Remarkably, extremal surfaces contain  many quadric configurations:

\begin{theorem}\label{TripleQuads} Any triple of skew lines on a smooth extremal surface determines a unique quadric configuration.
\end{theorem}

\begin{proof}[Proof]
Fix three skew lines, $L, L'$, and $L''$ on the extremal surface $X$ of degree $d=q+1$. Without loss of generality, assume
$X$ is defined by the  form $x^qw+w^qx+y^qz+z^qy$, $L$ by $x=y=0$, and $L'$ by $z=w=0$ (Corollary~\ref{DoubleStarTransCor}).
In this case,  $L''$ can be defined by linear equations of the form
$$
x= az+bw\qquad y = c z + d w,
$$ where the matrix 
$\begin{bmatrix}
a & b \\ c & d\end{bmatrix}
$
is full rank, and is parametrized as
$ \{ [as+bt:  c s + d t : s : t ] \;\; | \;\; [s: t] \in \mathbb P^1\}.
$
Furthermore, the condition that $L''$ lies on $X$ means that
$$
(as+bt)^qt + t^q(as+bt) + (c s + d t)^qs + s^q(c s + d t) = 0$$
for all $s, t$. This 
imposes the constraints
\begin{equation}\label{onX}
c^q+c = b^q + b = a^q+d = a+ d^q =0.
\end{equation}

The quadric $Q$ defined by
\begin{equation}\label{quadric}
c xz + d xw  -a y z -b yw
\end{equation}
contains $L$, $L'$, and $L''$. Note that $Q$ is the image of the 
 Segre map  
$$
\mathbb P^1 \times \mathbb P^1 \overset{\sigma}\rightarrow \mathbb P^3
\qquad
([s_1:s_2], [t_1, t_2]) \mapsto [(as_1+bs_2)t_1: (cs_1+ds_2)t_1 : s_1t_2:s_2t_2].
$$

Now, consider an arbitrary  line   in one of the rulings on $Q$,  say
$$
\ell= 
\{[(a\lambda_1 +b\lambda _2)t_1: (c\lambda_1+d\lambda_2)t_1 : \lambda_1t_2:\lambda_2t_2]\quad | \quad [t_1:t_2]\in \mathbb P^1\}.
$$
The line $\ell$ is on $X$ if and only if,  plugging   into the Frobenius form defining $X$, the form 
\begin{equation}\label{seg2}
\lambda_2(a\lambda_1 +b\lambda _2)^qt_1^{q}t_2 + 
\lambda_2^q(a\lambda_1 +b\lambda _2)t_1t_2^{q} 
+ \lambda_1(c\lambda_1+d\lambda_2)^qt_1^{q}t_2 +
 \lambda_1^q(c\lambda_1 +d\lambda _2)t_1t_2^{q},
\end{equation}
 is 
uniformly zero for all values of $t_1, t_2$.  Equivalently,  $\ell$ is on $X$ precisely when  the coefficients of $t_1^{q}t_2$ and of $t_1t_2^{q}$  in expression~(\ref{seg2}) satisfy
\[
\begin{aligned}
\lambda_2(a\lambda_1 +b\lambda _2)^q + \lambda_1(c\lambda_1+d\lambda_2)^q &=0\\
\lambda_2^q(a\lambda_1 +b\lambda _2) +  \lambda_1^q(c\lambda_1 +d\lambda _2)&=0.\\
\end{aligned}
\]
In light of the constraints~(\ref{onX}), these equations simplify to 
\begin{equation}\label{same}
\begin{aligned}
c\lambda_1^{q+1} + d\lambda_1^q\lambda_2 + a\lambda_1 \lambda_2^q + b \lambda_2^{q+1}&=0\\
\end{aligned}
\end{equation}
Because the form in (\ref{same}) is a  Frobenius form in $\lambda_1, \lambda_2$ with  the full rank matrix
$
\begin{bmatrix}
c & d\\
a & b
\end{bmatrix},
$
 there are precisely $q+1$ distinct solutions to (\ref{same}) in $\mathbb P^1$. We conclude that there are precisely $q+1$ lines $\ell$ of the form $\sigma([\lambda_1:\lambda_2]\times \mathbb P^1)$
 lying on both $X$ and $Q$. These are $q+1$ different skew lines on the extremal surface.

Now consider a line in the other ruling, say 
\[
m = \{[(as_1+bs_2)\lambda_1: (cs_1+ds_2)\lambda_1 : s_1\lambda_2:s_2\lambda_2] \;\; | \;\; [s_1:s_2]\in \mathbb P^1\}.
\]
 The line $m$  lies on $X$ if and only if 
 \[
\lambda_1^q \lambda_2 (as_1+bs_2)^qs_2 + \lambda_1\lambda_2^q(as_1+bs_2)s_2^q +  \lambda_1^q\lambda_2(cs_1+ds_2)^qs_1 +  \lambda_1\lambda_2^q(cs_1+ds_2)s_1^q = 0
\]
for all values of $s_1, s_2$. That is, $
m\subset X$ if and only if
\[
c^q\lambda_1^q\lambda_2 + c \lambda_1\lambda_2^q = 
a^q\lambda_1^q\lambda_2 +  d \lambda_1\lambda_2^q = 
a\lambda_1\lambda_2^q  + d^q \lambda_1^q\lambda_2 
=b^q\lambda_1^q\lambda_2 + b \lambda_1\lambda_2^q=0.
\]
Again making use of the relations~(\ref{onX}), these four equations all boil down to one,
\begin{equation}\label{only}
\lambda_1^q\lambda_2 -\lambda_1\lambda_2^q.
\end{equation}
Because there are exactly $q+1$ points $[\lambda_1:\lambda_2] \in \mathbb P^1$ satisfying (\ref{only}), 
there are precisely $q+1$ lines $m$ in this ruling of $Q$ which lie $X$. These form  a set of $q+1$ skew lines, each of which meets every line in the other set of $q+1$ skew lines on $X$.
 \end{proof}
 
\begin{remark}
In the finite geometry setting, Hirschfeld proves an analog of Theorem~\ref{TripleQuads} for Hermitian geometries using different techniques and language
\cite[19.3.1]{Hirschfeld}.
\end{remark}


\subsection{Symmetry of Quadric Configurations}

\begin{theorem} \label{TransQuad} The automorphism group of a smooth extremal surface  acts transitively on its set of  quadric configurations.
\end{theorem}

 In light of Theorem~\ref{TripleQuads},  Theorem~\ref{TransQuad} is an immediate consequence of the following:  
  \begin{theorem}\label{TripleTrans}
The automorphism group of a smooth extremal surface $X$  acts transitively on the set of triples of skew lines on $X$.
\end{theorem}

\begin{proof}[Proof of Theorem~\ref{TripleTrans}]
It suffices to show that $\Aut(X)$ acts transitively on the set $\mathcal S$  of all ordered sextuples $(L_1, L_2, L_3, M_1, M_2, M_3)$ of lines on $X$,
 consisting of two triples of skew lines $ \{L_1, L_2, L_3\}$ and $\{M_1, M_2, M_3\}$ with $L_i\cap M_j\neq \emptyset$ for all $i, j$. 
  
 Fix  an ordered  sextuple $(L_1, L_2, L_3, M_1, M_2, M_3)\in \mathcal S$. First note that its stabilizer,  even in $\PGL(4, k)$, is trivial. Indeed, the intersection points $p_{ij}=L_i\cap M_j$ must be fixed by any element in the stabilizer of $(L_1, L_2, L_3, M_1, M_2, M_3)$. These nine points contain five points in general linear position (no three on a line, no four on a plane). But an automorphism of $\mathbb P^3$ fixing five points in  general linear position is  trivial. 
 
 Next, we compute the cardinality of $\mathcal S$. There are  $(q^3+1)(q+1)$ choices for $L_1$ by Corollary~\ref{count}(a), and fixing $L_1$, there are $q^4$ choices for a skew line $L_2$ on $X$ by Theorem~\ref{starcount}(c). The number of choices for $L_3$ is the total number of lines on $X$ {\it minus} the number of lines meeting $L_1$ {\it or} $L_2$. Accounting for the double-counting of lines meeting {\it both} $L_1$ and $L_2$, the number of choices for $L_3$ is 
 \[
 \begin{aligned}
 \big[(&q^3+1)(q+1)\big] - 2\big[q^3+q+1\big] + \big[q^2+1\big] \\
 &= q(q^2+1)(q-1), 
\end{aligned}
\]
 using Corollary~\ref{count}(a), Theorem~\ref{starcount}(b), and Corollary~\ref{skewpair-connected}.
 The choice of the triple $L_1, L_2, L_3$ determines the quadric, and hence $q+1$ lines in $Q\cap X$ that all intersect $L_1, L_2, L_3$ by Theorem~\ref{TripleQuads}.
There are thus $(q+1)q(q-1)$ ways to choose the triple $M_1, M_2, M_3$. In total, the number of ordered sextuples is thus
\[
[(q^3+1)(q+1)]\cdot [q^4] \cdot [ q(q^2+1)(q-1)] \cdot [(q+1)q(q-1)] = q^6(q^4-1)(q^3+1)(q^2-1).
\]
This is precisely the order of the automorphism group $\Aut(X)$ by Corollary~\ref{orderAut}. So $\Aut(X)$ must act transitively on the set  $\mathcal S$, and hence on the set of all triples of skew lines on $X$.
\end{proof}

For future reference, we record the following corollary of the proof of Theorem~\ref{TripleTrans}:
 
\begin{corollary}\label{QuadricCount}
A smooth extremal surface $X$ contains exactly $\frac{1}{2}(q^3+1)(q^2+1)q^4$ quadric configurations,
where  the degree of $X$ is $q+1$. \end{corollary}

\begin{proof}[Proof of Corollary]
By Theorem~\ref{TripleQuads}, each quadric configuration on a smooth extremal surface $X$ is uniquely determined  by an ordered triple of skew lines $(L_1, L_2, L_3) $  on $X$. The number of such ordered triples is
$$
(q^3+1)(q+1) \cdot q^4 \cdot q(q^2+1)(q-1),
$$
as we computed in 
 the proof of Theorem~\ref{TripleTrans}. To determine the number of quadric configurations, then, we must determine the number of ordered triples determining the {\it same } quadric. To this end, first
note that there are $2(q+1)$ choices of a line $L_1$ in $\mathcal Q$.
Once $L_1$ is  fixed, the lines $L_2$ and $ L_3$ are among the $q$ lines in same ruling of $\mathcal Q$, so there are $q(q-1)$  choices for $(L_2, L_3)$.  We conclude that 
 there are \[\frac{(q^3+1)(q+1)q^5(q-1)(q^2+1)}{2(q+1)q(q-1)} = \frac{1}{2}(q^3+1)(q^2+1)q^4\] quadric configurations on a smooth extremal surface.
\end{proof}


\subsection{Star chords in Quadric Configurations} We record some observations about star chords and quadric configurations that will be useful in Section~\ref{Double2d}.

 \begin{lemma}\label{StarPointsQuadric}
 Let $Q$ be a quadric defining a quadric configuration on a smooth extremal surface $X$.
Let $\ell$ be a line on $Q$ but not on $X$. Then $\ell$  intersects $X$ in $q+1$ distinct points, and if any one of these intersection points is a star point of $X$, then they all are.
\end{lemma}

\begin{proof}
Because the automorphism group of $X$ acts transitively on quadric configurations (Theorem~\ref{TransQuad}), we may assume that $X$ is given by the Fermat Frobenius form and $Q$ by $xw=yz$.
The lines on $Q$ have the following parametrizations
$$
\{[\lambda s: s : \lambda t : t ]\;\; | \;\; [s:t]\in \mathbb P^1\} \quad {\text{ and }} \quad 
\{[\lambda s: \lambda t :s :   t ]\;\; | \;\; [s:t]\in \mathbb P^1\}. 
$$
Without loss of generality, let $\ell = \{[\lambda s: s : \lambda t : t ]\,\,\, | \,\,\,\, [s:t]\in \mathbb P^1\}$ for some fixed $\lambda$.
The condition that a point  $[\lambda s_0: s_0 : \lambda t_0 : t_0]$
of  $\ell$ lies on $X$ is that
\begin{equation}\label{starLineEq}
(\lambda s_0)^{q+1}+ s_0^{q+1} +( \lambda t_0)^{q+1} + t_0^{q+1} = (\lambda^{q+1}+1)(s_0^{q+1}+t_0^{q+1}) = 0.
\end{equation}
There are two ways this can happen. Either  $\lambda^{q+1}=-1$, which means (\ref{starLineEq}) holds for all values of $[s_0:t_0]$, so  the line $\ell$ lies on $X$. 
Or $\lambda^{q+1}\neq -1$, and there are exactly $q+1$ points $[s_0:t_0]$ satisfying  $s_0^{q+1}+t_0^{q+1}=0$. In this case, there are exactly $q+1$ distinct points of $\ell\cap X$, all of the form $[\lambda \mu : \mu : \lambda : 1]$ where $\mu$ ranges through the $q+1$ distinct roots of $-1$.  In particular,  $\mu\in \mathbb F_{q^2}$. 
Now if  one of these points $[\lambda \mu : \mu : \lambda : 1]$ is a star point, then it is defined over $\mathbb F_{q^2}$ (Proposition~\ref{HermStar}),  so $\lambda\in \mathbb F_{q^2}$ as well. Thus all $q+1$ points of $X\cap \ell$ are defined over  $\mathbb F_{q^2}$ and  hence all are star points. \end{proof}

\begin{proposition} \label{QuadricToSkew} Let $Q$ be a smooth quadric defining a quadric configuration on a smooth  extremal surface $X$. Then there are exactly $q^2-q$ star chords in each ruling of $Q,$ and those in opposite rulings meet off $X$.
 \end{proposition}

\begin{proof}
Consider a star chord $\ell$ on $Q$. Write $Q\cap X =\mathcal L \cup \mathcal M$ where $\mathcal L$ and $\mathcal M$ are the two skew sets of lines on $X$ in opposite rulings of $Q$.

Because $\ell$ must lie in one of the rulings of $Q$, it  intersects each of the $q+1$  lines in,  say, $\mathcal M$. For each  $M\in \mathcal M$, the intersection point
 $\ell \cap M$  is a star point (Theorem~\ref{StarLines}(iv)). Conversely, through each  star point on $M$, the unique line in the opposite ruling of $Q$ is either a line in  $\mathcal L$, or a star chord, depending on whether or not it  is on $X$ (Lemma~\ref{StarPointsQuadric}).
Since there are  $q^2+1$ total star points on $M$ (Theorem~\ref{starcount}(a)), this leaves 
 $q^2-q$ possible points of intersection of the star chord $\ell$ with $M$. Thus there are exactly $q^2-q$ possibilities for the star chord $\ell$ in this ruling of $Q$. By symmetry, the same holds in the other ruling.

Now suppose $\ell$ and $m$ are star chords in opposite rulings on $Q$.
If $p=\ell \cap m$ lies on $X$, then it must be one of the $q+1$ points on $\ell\cap X$, and hence $p$ is  some star point on some line $M\subset Q\cap X$ in the ruling opposite $\ell$. In this case,  $M$ is the unique line through $p$ on $Q$ in the ruling opposite $\ell$, forcing 
$m=M$. This contradicts our assumption that $m$ is not on $X$.
\end{proof}

\begin{remark}\label{QuadricToSkew2}
Proposition~\ref{QuadricToSkew} and Lemma~\ref{StarPointsQuadric} together say the complete set of lines on $Q$ passing through star points of $X$ consists of two sets of $q^2+1$ skew lines (one on each ruling); in each of these skew sets, there are $q+1$ lines on $X$ and $q^2-q$ star chords. 
\end{remark}

\begin{theorem} \label{TransQ2l} The automorphism group of a smooth extremal surface acts transitively on the set of triples $(Q, \ell, m)$ consisting of 
 a quadric $Q$ defining a quadric configuration, together with a choice star chords  $ \ell$ and $ m$, one  in each ruling of $Q$. 
\end{theorem}

\begin{proof}
We may assume that the extremal surface $X$ is defined by $x^{q+1}+y^{q+1}+z^{q+1}+w^{q+1}$ and $Q$ by $xw-yz$ (Theorem~\ref{TransQuad}). Let $\ell $ and $m$ be an arbitrary
 pair of star chords on $Q$, lying in opposite  rulings.
 It suffices to show that there is an automorphism of $X$ which stabilizes $Q$ and
 sends $\ell $ and $m$ to the star chords $\mathbb V(x, z)$ and to $\mathbb V(z, w),$ respectively.
 
 The lines in the two rulings of  $Q$ have the form
\[
 \mathbb V(\lambda x - \mu y, \lambda z - \mu w)   \,\,\,\,\,\,\,\,\,\, {\text{and}} \,\,\,\,\,\,\,\,\,\,  \,\, \mathbb V(\alpha x - \beta z, \alpha y - \beta w);
\]
the star chords among them are precisely those where $[\lambda:\mu]$ (respectively  $[\alpha :\beta]$)  is an $\mathbb F_{q^2}$ point of $\mathbb P^1$ not on $\mathbb V(s^{q+1}+t^{q+1})$. Indeed, all such lines are on $Q$,  but not on $X$, and since there are $q^2-q$  in each ruling, we have found the complete list of star chords on $Q$  (Proposition~\ref{QuadricToSkew}).
 
 Suppose that $\ell = \mathbb V(\lambda x - \mu y, \lambda z - \mu w).$
The change of coordinates  $g$ where
$
g^{-1}= \begin{bmatrix}
\lambda^q & \mu & 0 & 0 \\
-\mu^q & \lambda &0& 0\\
 0 & 0 & \lambda^q & \mu  \\
 0& 0 & -\mu^q & \lambda\\
\end{bmatrix}
$ is in  $ \Aut (X) \cap \Aut (Q)$, since it 
simply scales the defining equation of both $X$ and $Q$ by a non-zero scalar (remember $\lambda^{q+1}+\mu^{q+1} \neq 0$).
In addition,  $g$ sends    $\ell$ to $\mathbb V(x, z)$, as
\[
\begin{aligned} g(\ell) &= \mathbb V\left(\lambda (\lambda^q x + \mu y) - \mu (-\mu^{q} x + \lambda y), \, \lambda (\lambda^qz +  \mu  w) - \mu (-\mu^q z + \lambda w)\right)\\
& \ =\mathbb V\left((\lambda^{q+1}+\mu^{q+1})x, \, (\lambda^{q+1}+\mu^{q+1})z\right)
\, =\,  \mathbb V(x, z).
\end{aligned}
\]

Of course, $g$ sends $m$ to some star chord on $Q$ in the opposite ruling from $g(\ell)$. So $g(m) = \mathbb V(\alpha x - \beta z, \alpha y - \beta w)$ for some
 $\mathbb F_{q^2}$ point $[\alpha:\beta]\in \mathbb P^1$ not on $\mathbb V(s^{q+1}+t^{q+1})$. 
 Now observe that the change of coordinates $h$ where 
$h^{-1}= 
 \begin{bmatrix}
 \beta & 0 &\alpha^q &0\\
0 & \beta & 0 &\alpha^q \\
\alpha &0 &- \beta^q & 0 \\
0&\alpha &0 &- \beta^q \\
\end{bmatrix}
$ preserves the Fermat extremal surface and  the quadric $Q$ defined by $xw=yz$. In addition,  $h$  preserves the line $V(x, z)$, since
$h(x)$ and $h(z)$ are forms in only $x$ and $z$. Finally,  
the line $g(m)=\mathbb V(\alpha x - \beta z, \alpha y - \beta w)$ is sent to 
\[
\begin{aligned}
h(g(m)) 
 = \ & \mathbb V(\alpha (\beta x+\alpha^q z)- \beta (\alpha x -\beta^q z),\,  \alpha (  \beta y + \alpha^q w) - \beta ( \alpha y -\beta^qw))\\
 = \ & \mathbb V( (\alpha^{q+1}+\beta^{q+1}) z, \,  (\alpha^{q+1}+\beta^{q+1})w)\,
 = \, \mathbb V(z, w)
 \end{aligned}
 \]
 We conclude that the composition $h\circ g$ is an automorphism of $X$ which preserves $Q$, and takes
 $\ell$ and $m$ to $\mathbb V(x, z)$ and $\mathbb V(z, w)$, respectively. This completes the proof.
 \end{proof}

\begin{corollary}\label{dualstarlinesinquadric}
If a star chord $\ell$ is in a quadric $Q$  defining a quadric configuration on a smooth extremal surface, then its dual star chord $\ell'$ (Definition~\ref{dual}) is also on $Q$, necessarily in the same ruling as $\ell$.
\end{corollary}
\begin{proof}
Assume that the extremal surface  is the Fermat surface, and $\ell$  is the star chord $\mathbb V(x, y)$ 
on the quadric $Q=\mathbb V(xw-yz)$ (Theorem~\ref{TransQ2l}). The dual chord of $\ell$ is $\ell'=\mathbb V(z, w)$ (Example~\ref{standardStarLineConfig}),  which clearly lies on $Q$ as well. The lines $\ell$ and $\ell'$ lie in the same ruling of $Q$ because they are skew (Theorem~\ref{StarLines}(ii)).
\end{proof}

\begin{corollary}\label{quadrics intersecting in star chords} Let $\ell$ and $m$ be star chords for a smooth extremal surface $X$. If $\ell$ and $m$ lie on a quadric  that defines a quadric configuration on $X$, then $\ell$ and $m$ lie on exactly  $q+1$ quadrics  that define quadric configurations on $X$.
\end{corollary}

\begin{proof}

By Theorem~\ref{TransQ2l}, we can assume that $X$ is the Fermat extremal surface, $\ell=\mathbb V(x, y)$ and $m=\mathbb V(x, z)$.  The quadrics defining quadric configurations that contain $\ell$ and $m$ must also contain their dual star chords,  $\ell'=\mathbb V(z, w)$ and $m'=\mathbb V(y, w)$, respectively, by Corollary~\ref{dualstarlinesinquadric}. 

The quadrics containing  $\{\ell,\ell', m, m'\}$ are  defined by degree two polynomials in the ideal 
$$
\langle x, z \rangle \cap 
\langle y,w \rangle \cap 
\langle  z, w \rangle 
\cap \langle x, y \rangle =
\langle xw, yz \rangle.
$$ 
But a quadratic form $\mu\ xw - yz$ (where $\mu$ is a non-zero scalar)  
 defines quadric containing lines of $X$ if and only if  $\mu^{q+1}=1$. 
Indeed, 
the lines in one of the rulings are parametrized by $[a:b]\in \mathbb P^1$:
$$
L_{ab} = \{[as: \mu bs: at: bt] \,\, | \,\, [s:t]\in \mathbb P^1\},
$$
which lies on  the Fermat surface only if $\mu^{q+1}=1$ and $a^{q+1}+b^{q+1}=0$.
 Thus there are $q+1$ quadrics  that contain the four star chords $\{\ell,\ell', m, m'\}$.
\end{proof}




 
   
\section{\texorpdfstring{Double $2d$ Configurations}{Double 2d Configurations}}\label{Double2d}

One fascinating classical feature of the geometry of a cubic surface is the existence of thirty six ``double sixes'' \cite{Schlafli}.
 A double six consists of two collections of six skew lines on the cubic, with the property  that each line in one collection intersects exactly five lines in the other. A choice of double six is equivalent to a labeling of the twenty-seven lines on the cubic so that one of the collections of six skew lines is the set of six exceptional divisors, thinking of the cubic surface as the blow up of the plane at six points, and the other collection is the set of  strict transforms of the six  conics through five of the points. In this section, we present a generalization of a ``double six'' which exists on all extremal surfaces.

\begin{definition}\label{double}
For any $d \geq  2$, a {\bf double 2d} is a collection of two sets, $\mathcal{A}$ and $\mathcal{B}$, each consisting  of $2d$ lines  in projective three space, such that
\begin{enumerate}
\item Each line in $\mathcal{A}$ (resp. $\mathcal{B}$) is skew to every other line in $\mathcal{A}$ (resp. $\mathcal{B}$); and
\item Each line in $\mathcal{A}$ (resp. $\mathcal{B}$) intersects exactly $d+2$ lines in $\mathcal{B}$ (resp. $\mathcal{A}$). 
\end{enumerate}
\end{definition}

  Typically, we do not an expect a surface of degree $d$ to contain {\it any} double $2d$---for example, a general surface in $\mathbb P^3$ of degree greater than three contains no line \cite[12.8]{Harris}. The next result guarantees, however, that like cubic surfaces, extremal surfaces always contain double $2d$'s.  
 
  \begin{theorem}\label{Quadric2}
  Every smooth extremal surface of degree $d$ contains  double $2d$ configurations of lines. 
  \end{theorem}

In fact, there a great many double $2d$'s on an extremal surface:   Corollary~\ref{lowerBound} will eventually  show that their number grows asympotically to  $\frac{1}{16}d^{14}$ as $d$ grows large.

We will prove Theorem~\ref{Quadric2} by constructing  explicit pairs of quadric configurations   whose union is a double $2d$. First, we speculate that every double $2d$ arises from pairs of quadrics:

\begin{conjecture}\label{DoubleConj}
Every double $2d$ on an extremal surface $X$ of degree $d$ consists of $4d$ lines that are the union of two quadric configurations on $X$. 
\end{conjecture}

Towards Conjecture~\ref{DoubleConj}, we have proven
\begin{theorem}\label{ProgressConj}
  Every double $2d$ on  a degree $d$ extremal surface is a union of two quadric configurations when  $d>10$ or $d<5$. Moreover, for $d\geq 5$, if two quadrics determine some double $2d$, then no other pair of quadrics determines the same  double $2d$. \end{theorem}
We have also verified the conjecture by computer when $d=5$.

Before proving Theorems~\ref{Quadric2} and~\ref{ProgressConj}, we review the motivating example of cubic surfaces.

\subsection{Double Sixes on Cubics}\label{Cubics}
Every double six on a cubic surface---whether extremal or not---is a union of two quadric configurations.
For an arbitrary double six $\mathcal A \cup \mathcal B$ on a cubic surface $X$, there is a choice of coordinates making $X$ the blowup of six points on $\mathbb P^2$  (no three on a line, not all on a conic), 
and so that $\mathcal A$ consists of the six lines of exceptional divisors $\{E_1, \dots, E_6\}$ and $\mathcal B$  consists of the proper transforms $\{\tilde C_1, \dots, \tilde C_6\}$ of the six conics  in $\mathbb P^2$ through five of the six points \cite[Thm 8, pg. 366]{CubicBlowUp}. Here, $\tilde C_i$ denotes the proper transform of the conic that misses the point  blown up to get $E_i$.

Now, given any three lines in $\mathcal A$, say  $\{E_1, E_2, E_3\}$, there are  three lines, $\{\tilde C_4, \tilde C_5, \tilde C_6\}$,  in $\mathcal B$ that meet all of them. This says that the unique quadric surface $Q$ containing  $\{E_1, E_2, E_3\}$ must also contain   $\{\tilde C_4, \tilde C_5, \tilde C_6\}$. Likewise, the unique quadric  $Q'$ containing  $\{E_4, E_5, E_6\}$ must  contain   $\{\tilde C_1, \tilde C_2, \tilde C_3\}$.  So the quadrics $Q$ and $Q'$ both produce quadric configurations on $X$:
  $$
  Q\cap X = \{E_1, E_2, E_3, \tilde C_4, \tilde C_5, \tilde C_6\}\,\,\,\,\,\,\,\,\,\,\,\,\,\,\,\,\,\,\,\,
  {\text{and}} \,\,\,\,\,\,\,\,\,\,\,\,\,\,\,\,\,\,\,
   Q'\cap X = \{E_4, E_5, E_6, \tilde C_1, \tilde C_2, \tilde C_3\},
   $$ which together produce the double six
    $$ \mathcal A = \{E_1, E_2, E_3, E_4, E_5, E_6\}
    \,\,\,\,\,\,\,\,\,\, \,\,\,\,\,\,\,\,\,\,   {\text{and}} \,\,\,\,\,\,\,\,\,\,\,\,\,\,\,\,\,\,\, \mathcal B = \{\tilde C_1, \tilde C_2, \tilde C_3,  \tilde C_4, \tilde C_5, \tilde C_6\}.
    $$
    So every double six on a cubic surface is the union of two quadric configurations.
   
 \begin{remark}\label{NotUnique} The quadric configurations $Q$ and $Q'$ determining the double six $\mathcal A\cup \mathcal B$ on a cubic surface are not unique: there is a quadric containing any three of the six skew lines in $\mathcal A$ and another containing the remaining three, and the lines of $\mathcal B$ lie three in each of these two quadrics. Thus there are $\frac{1}{2}\binom{6}{3} = 10$ different pairs of quadrics determining the double six $\mathcal A\cup \mathcal B$.  This confirms that {\it some} restriction on $d$ is necessary in the uniqueness statement in Theorem~\ref{ProgressConj} above.
 \end{remark}
  
\begin{remark}\label{cubicstarlines}
The previous discussion applies to an arbitrary smooth cubic surface: each of its thirty-six double sixes is a union of two quadric configurations. 
However, for an {\it  extremal} cubic surface, the double sixes come from two quadrics of a particular form.  Specifically,  if $Q$ and $Q'$ are quadrics  on an extremal cubic surface which together give a double six, then $Q \cap Q'$ is the union of four lines.

To see this, observe that $Q\cap Q'\cap X$ consists of twelve {\it distinct } points---otherwise, one line of $Q\cap X$ would intersect a line from both rulings of $Q'\cap X$ (or vice versa), violating the skewness  condition for  a double six. These twelve points are star (Eckardt) points as they lie at the intersection of a line in $Q \cap X$ with a line in $Q' \cap X$.
Now, each of these twelve  star points lies on only one line in $X\cap Q$, again by skewness, so these twelve star points lie  on star chords of $Q\cap X$. Since there are only two  star chords in each ruling (Remark~\ref{QuadricToSkew2}), each containing exactly $q+1=3$ star points, these twelve  points  lie three each  on the four star chords on $Q$. Likewise, the same argument replacing $Q$ by $Q'$ shows that the twelve star points lie three each  on the four star chords on $Q'$. We conclude that $Q\cap Q'$ consists of the  four  shared star chords for $X$.
\end{remark}


\subsection{\texorpdfstring{The existence of double $2d$'s on extremal surfaces}{The existence of double 2d's on extremal surfaces}}
\begin{proof}[Proof of Theorem~\ref{Quadric2}]
Choose coordinates so that  the extremal surface $X$ is defined by $\fermat$.  

Fix $\mu$,  a $(q+1)$-st  root of unity. As we saw in Example~\ref{exQuadric2}, 
 the lines 
$$\mathcal L_{\mu} := \{\mathbb V(x-\alpha y,  z-\mu \alpha w) \,\, \,| \,\,\, \alpha^{q+1} = -1\}$$
 and $$\mathcal M_{\mu} =  \{\mathbb V(x-\beta z,  y-\mu \beta w) \,\, \,\, |\,\,\, \beta^{q+1} = -1 \}$$
 form a quadric configuration cut out by the quadric $Q_{\mu}= \mathbb V(\mu xw -yz)$.

We claim that if  $\mu_1$ and $ \mu_2$ are  {\it distinct} $(q+1)$ roots of unity, then the sets
$$\mathcal{A} := \mathcal L_{\mu_1} \cup  \mathcal M_{\mu_2}
\,\,\,\,\,\,{\text{and}} \,\,\,\,\,\,\,\,
 \mathcal{B} := \mathcal L_{\mu_2} \cup \mathcal M_{\mu_1}$$ {\it {  together form a double $2(q+1)$}.  }
 
 To see that $\mathcal A$ consists of skew lines,  first observe that the  lines of  $ \mathcal L_{\mu_1}$  are mutually skew, as they lie in the same ruling of a quadric. To see that  each $L\in  \mathcal L_{\mu_1}$ is  skew to  every  $M\in  \mathcal M_{\mu_2}$,  we check  that 
  the ideal of their intersection, 
$\langle x-\alpha y,  \ z-\mu_1\alpha  w, \  x-\beta z,  \ y-\mu_2 \beta w\rangle,$
is generated by four linearly independent linear forms. For this, it suffices to show that the matrix
$$
\begin{bmatrix}
1 & -\alpha & 0 & 0 \\
0 & 0 & 1 & -\mu_1 \alpha \\
1 & 0 & -\beta & 0 \\
0 & 1 & 0 & -\mu_2 \beta\\
\end{bmatrix}, 
$$
whose rows are the coefficients of the linear forms, has full rank. But this is clear, since its determinant is 
$\alpha \beta (\mu_2-\mu_1)$.
A symmetric argument shows that also $\mathcal B$ consists  of skew lines.

Now that we know $\mathcal A$ and $\mathcal B$ are skew sets, 
the proof of Theorem~\ref{Quadric2} will be complete once we have proved the following general lemma. 

\begin{lemma}\label{skew implies double}
Let $X$ be a smooth extremal  surface of degree $d$. Let  $\mathcal{Q}_1$ and  $\mathcal{Q}_2$ be two quadric configurations on $X$ that do not share a line (on $X$). Write $\mathcal{Q}_1 = \mathcal{L}_1 \cup \mathcal{M}_1$ and $\mathcal{Q}_2 = \mathcal{L}_2 \cup \mathcal{M}_2$ for  the decomposition of each quadric configuration into the lines of the two rulings. 
Then   $\mathcal{A} = \mathcal{L}_1 \cup \mathcal{M}_2$ and $\mathcal{B} = \mathcal{L}_2 \cup \mathcal{M}_1$  form a double $2d$ on $X$ if (and only if) both $\mathcal{A}$ and $\mathcal{B}$ are skew sets.
\end{lemma}

\begin{proof}[Proof of Lemma~\ref{skew implies double}]
Since $\mathcal{Q}_1$ and  $\mathcal{Q}_2$ have no common line,  there are $4d$ lines in $\mathcal Q_1\cup \mathcal Q_2$, and $2d$ lines in each of $\mathcal A$ and $\mathcal B$.   Because we are given that 
 $\mathcal{A}$ and $\mathcal{B}$ are each skew sets,   we need only check condition (2) of Definition~\ref{double} to verify that  $\mathcal{A}\cup \mathcal{B}$ is a double $2d$.
 
 To this end, take any $N\in \mathcal A$. Without loss of generality, assume $N\in \mathcal L_1$. We need to show that $N$ intersects exactly $d+2$ lines in $\mathcal B$. Since $N$ lies in one ruling of the quadric $Q_1$ determining $\mathcal Q_1$,    the line $N$ intersects the $d$ lines of the opposite ruling $\mathcal M_1\subset \mathcal B$. Thus we need to show that 
 $N$ intersects exactly two  lines of $\mathcal L_2$.
  
  Since $N$ does not lie on the quadric $Q_2$ determining $\mathcal Q_2$ (remember $\mathcal Q_1\cap \mathcal Q_2=\emptyset$), its intersection multiplicity with $Q_2$ is two. If   
 $N$ meets $Q_2$ in two distinct points, we are done:  $N$ must meet exactly  two of the lines in the  ruling $\mathcal L_2$ since it does not meet any line of the ruling $\mathcal M_2$ by our assumption that $\mathcal A$ is a skew set. 
 
It remains to show that $N$ can not be tangent to $Q_2$. If, on the contrary,  $N$ is tangent to $Q_2$ at some point $p$, then  $N\subset T_pQ_2$.
 Because $p\in Q_2\cap X$, and $Q_2\cap X$ is a union of lines, the point $p$   lies  on some line $M$ in    $Q_2\cap X$. In particular,  $p$ is a star point since  it is the intersection of the two lines  $M$ and $N$
 on $X$. 
Furthermore,  since both $N$ and $M$ are in 
the tangent plane $T_pQ_2$, as well as in the star plane $T_pX$, we have $T_pX=T_pQ$. But now consider the unique line $\ell$ through the star point $p$ on  $Q_2$ in the opposite ruling from  $M$.
We know $\ell$ is not on $X$, for otherwise, $p\in \ell\subset Q_2\cap X$, which means  $p$ lies on lines  in {\it both} rulings of $\mathcal Q_2$, violating skewness.
By Lemma~\ref{StarPointsQuadric}, we conclude that  $\ell$ is a star chord  through  $p$, and being on $Q_2$, also $\ell\subset T_pQ_2=T_pX$. But no star chord through a star point $p$ can lie in the star plane $T_pX$ (Remark~\ref{caution}). This contradiction ensures that  $N$ is not tangent to $Q_2$, and the proof is complete.  \end{proof}
\end{proof}


\subsection{Pairs of Quadrics containing a common line}
  The double $2d$ constructed in the proof of Theorem~\ref{Quadric2} is obtained from two quadric configurations  whose quadric surfaces intersect in  four lines.
  These are an abundant type of double $2d$'s---encompassing all the double sixes in the case of extremal cubics.

\begin{theorem}\label{double2d-star-line} Let $X$ be a smooth extremal surface of degree $d$, and let 
 $Q$ and $Q'$ be distinct quadrics defining quadric configurations on $X$.

Assume that $Q$ and $Q'$   share a common line, but share no line on  $X$.
Then
\begin{enumerate}
\item[(i)]  The $4d$  lines of $ (Q\cap X) \cup (Q'\cap X)$ can be split into two sets of $2d$ lines forming a double $2d$;
\item[(ii)]
The intersection $Q\cap Q'$ consists of four star chords  $\{\ell, m, \ell', m'\}$, where $\{\ell, \ell'\} $ and
$\{m,  m'\} $ are dual chord pairs in opposite rulings.
\end{enumerate}
\end{theorem}

  Importantly, not all double $2d$'s are of the type guaranteed by Theorem~\ref{double2d-star-line}; see Example~\ref{double2d-not-star-line}.

Before proving Theorem~\ref{double2d-star-line}, we deduce the following corollary bounding below the total number of  double $2d$s on an extremal surface.

\begin{corollary}\label{lowerBound}
An extremal surface of degree $d = q+1\geq 5 $ contains at least 
$$
\frac{1}{16}(q^3+1)(q^2+1)(q-1)^2q^7  
$$
collections of double $2d$'s.
\end{corollary}

\begin{proof}[Proof of Corollary]
By Theorem~\ref{ProgressConj}, if two quadrics determine a double $2d$ on an extremal surface of degree $d\geq 5$, then they are unique. So  we can  prove  Corollary~\ref{lowerBound} by counting the 
pairs of quadrics $\{Q, Q'\}$ determining quadric configurations whose intersection consists of four star chords (Theorem~\ref{double2d-star-line}).

Fix one quadric $Q$  giving a quadric configuration on $X$.
There are $(q^2-q)^2$ choices of pairs of star chords $\{\ell, m\}$ on $Q$, one in each ruling, by Proposition 
\ref{QuadricToSkew}. Since the dual of each star chord on $Q$ is also on $Q$,
there are $\frac{(q^2 - q)^2}{4}$ choices for sets of star chords $\{\ell,\ell', m, m'\}$ on $Q$, where 
$\ell$ and $m$ are  in opposite rulings  and $\ell', m'$ are their duals. 

There are exactly $q$ additional quadrics, besides $Q$, that contain $\{\ell,\ell', m, m'\}$ and define a quadric configuration (Corollary~\ref{quadrics intersecting in star chords}). So  there are exactly $\frac{q^3(q-1)^2}{4}$ quadrics $Q'$ defining quadric configurations such that $Q\cap Q'$ is the union of two star chords and their duals.

Finally, multiplying by the total number of choices for $Q$ (provided by Corollary~\ref{QuadricCount}), we get 
\[\frac{1}{2}(q^3+1)(q^2+1)q^4\cdot\frac{1}{4}q^3(q-1)^2 = \frac{1}{8}(q^3+1)(q^2+1)(q-1)^2q^7\]
 {\it ordered}  pairs of quadric configurations whose intersection is four star chords.
This counts each pair twice so the result follows.
\end{proof}

\begin{proof} [Proof of Theorem~\ref{double2d-star-line}]
Suppose $\ell \subset Q  \cap Q'$ but $\ell \not\subset X$.  Because $\ell$ is in some ruling on each of $Q$ and $Q'$, $\ell$ must intersect $d$ lines on $X \cap Q$ and $d$ lines on $X \cap Q'$.  By hypothesis, these lines are distinct, so $\ell$ intersects $2d$ lines on $X$.  Now because  $\ell \cap X$ can be at most $d$ points, $\ell$  simultaneously intersects $X$ at  a line on $X \cap Q$  and  a line on $X\cap Q',$ so $\ell$ intersects $X$ at a star point.   So $\ell$ passes through $d$ star points and is a star chord. 

Let $\ell'$ be the dual star chord to $\ell$. We know $\ell' \subset Q \cap Q'$, by  Corollary~\ref{dualstarlinesinquadric}.  Since $\ell$ and $\ell'$ are  skew, they are in the same ruling on $Q$ and also  in the same ruling on $Q'$, which means that $\ell \cup \ell'$ is a curve of bi-degree $(2,0)$ on each quadric. 
Since  $Q \cap Q'$ is a curve of bidegree $(2,2)$ on each quadric, the residual intersection curve has bidegree $(0,2)$ in each quadric. Since homogeneous polynomials in two variables over an algebraically closed field factor into linear terms, this residual curve is either two distinct lines, or a double line.  In particular, it contains some line $m$, which, by the argument above, must be a star chord. Now again by Corollary~\ref{dualstarlinesinquadric}, the residual intersection must be two dual star chords $m$ and $m'$. This proves (ii).
 
 To prove (i), we use Theorem~\ref{TransQ2l} to chose coordinates so that $X$ is the Fermat extremal surface, $Q$ is the quadric defined by $xw=yz$, and $\ell$ and $m$ are the lines $\mathbb V (x, z)$ and $\mathbb V(w, z)$,
 respectively.  In this case, we have already computed  (in the proof of Corollary~\ref{quadrics intersecting in star chords})  that 
 the quadrics containing $\{\ell,  \ell', m, m'\}$ and defining quadric configurations are all 
of the form $\mathbb V(\mu xw - yz) $ where $\mu^{d}=1$, and that any two such quadrics define a double $2d$ (in  the proof of Theorem~\ref{Quadric2}).
\end{proof}

\begin{remark}\label{d=3}
The bound in Corollary~\ref{lowerBound} is not valid when $d$ is less than $5$ because in this case, there can be multiple pairs of quadric configurations that determine the same double $2d$. For example, every double six on a cubic surface  can be split into the union of two quadric configuration in ten different ways 
(Remark~\ref{NotUnique}).  Note that dividing the bound provided by Corollary~\ref{lowerBound} by ten, we get a lower bound of $ 36$ double sixes on a cubic surface,  recovering the fact that {\it all} double sixes on an extremal cubic comes from quadrics sharing star chords (Remark~\ref{cubicstarlines}).

Similarly, when $d = 4$, there are double eights that  split into the union of two quadric configurations  in multiple ways.
For example, the double eight on the Fermat quartic defined by the two quadrics 
$Q_1=\mathbb V(xw-yz)$ and 
$Q_2=\mathbb V(xw+yz)$ can also be given by two different quadrics $Q_3$ and $Q_4$,
 as one can check by examining the intersection matrix for the sixteen lines of $(Q_1\cap X)\cup (Q_2\cap X)$ to find a different grouping into lines in two quadrics.
\end{remark}

\begin{example}\label{double2d-not-star-line} 
We now construct an  example 
of a double eight on a quartic extremal surface that can not be given two quadrics sharing a line. This shows that  not every double $2d$ on an extremal surface is of the special type in Theorem~\ref{double2d-star-line}. 

We work on the Fermat quartic, $X = \mathbb{V}(x^4+y^4+z^4+w^4)$ in characteristic three.
The quadrics
$
 Q_1  = \mathbb{V}(xw-yz) $
 and 
 $ Q_2  = \mathbb V(x^2+xy+xz-xw-y^2+yz+yw+z^2-zw-w^2)$
 both give quadric configurations on $X$. 
The quadric configuration $X\cap Q_1$ is the union $\mathcal{L}\cup \mathcal{M}$ where 
\[\mathcal{L} = \{\mathbb{V}(x-\alpha y,z-\alpha w)\,\, | \,\, \alpha^4=-1\}\,\,\,
\text{ and }\,\,\,
\mathcal{M} = \{\mathbb{V}(x-\alpha z,y-\alpha w) \,\, | \,\, \alpha^4=-1\},\]
as we computed in Example~\ref{exQuadric2}.
The quadric configuration $X\cap Q_2$ is the union $\mathcal{N}\cup \mathcal{P}$ where 
\[\mathcal{N}=\{\mathbb{V}(x-aw,y-az), \mathbb{V}(x-\overline{a}w,y-\overline{a}z), \mathbb{V}(-x-y+w,x-y-z), \mathbb{V}(-x-y-w,x-y+z)\}\text{ and }\]
\[\mathcal{P} = \{\mathbb{V}(x+ay,z-\overline{a}w), \mathbb{V}(x+\overline{a}y,z-aw),\mathbb{V}(-x+y+w,-x-y+z), \mathbb{V}(-x-y+w,x-y+z)\},\]
where $a$ and $\overline{a}$  are the roots in $k$ of the polynomial $T^2-T-1$ over $\mathbb F_3$.
We leave it to the reader to directly verify these eight lines all lie on both $Q_2$ and $X$.

The set    $\mathcal  A\cup \mathcal B$ is a double eight, where 
$\mathcal A = \mathcal L \cup \mathcal N $ and $\mathcal B = \mathcal M\cup  \mathcal P$. To check this, 
it suffices to check that the lines in $\mathcal L$ and skew to those in $\mathcal N$, and similarly
that the lines in $\mathcal M$ are skew to those in $\mathcal P$ (Lemma~\ref{skew implies double}), which can be directly verified. 
 
It remains to check that $Q_1$ and $Q_2$ do not share any line. If they did, then there are two shared lines in each ruling (Theorem~\ref{double2d-star-line}). So it suffices to show an arbitrary  line $\ell=\{\lambda s: s: \lambda t : t] \,\,\, [s:t]\in \mathbb P^1\}$ in one of the rulings of $Q_1$ can not lie on $Q_2$. If $\ell\subset Q_2$, then the points $[0:0:\lambda:1]$,  $[\lambda:1:0:0]$
and $[\lambda:1:-\lambda:-1]$ in $\ell$ must all 
 lie on $Q_2$. Plugging into the equation for $Q_2$ produces the constraints
\[
\lambda^2 -\lambda - 1  = 0, \,\,\,\,\,\lambda^2 +\lambda - 1  = 0,  \,\,\,\,\, {\text{and}}\,\,\,\,\,\, \lambda^2=0.
\]
Because these three equations are inconsistent, we conclude that $\ell$ does not lie on $Q_2$. 

Finally, we must show that the double eight $\mathcal A \cup \mathcal B$ can not be given by {\it any other} pair of quadrics that {\it do} share a line. To this end, assume on the contrary  that 
$\mathcal A \cup \mathcal B$ is  given by quadrics $Q_3$ and $Q_4$, and that $\ell \subset Q_3\cap Q_4$
for some line $\ell$. Furthermore, since $Q_1$ and $Q_2$ share no line, we may assume that $\ell\not\subset Q_1$;  in particular, $\ell$ intersects two lines in each ruling of $Q_1$.
Because $\ell$ lies in  one ruling of  each of  $Q_3$ and of $Q_4$, \  $\ell$ meets each in a set of 
 {\it eight} skew  lines in the double eight $(Q_3\cap X)\cup (Q_4\cap X) = \mathcal A \cup \mathcal B$. Since at most two of these eight intersection points are on $Q_1$, we know $\ell$ intersects at least six of the lines in $Q_2$, and so $\ell \subset Q_2$. But this is impossible:  $\ell$ lies in one of the rulings of $Q_2$ (and is not on $X$), so it intersects exactly four of the lines on $Q_2\cap X$.
\end{example}

\subsection{Progress towards Conjecture~\ref{DoubleConj}}\label{Progress}
We now prove Theorem 
\ref{ProgressConj}.
The proof will mainly use the combinatorics of the intersection matrix between the two skew sets of size $2d$ and the properties of quadrics. 

 \begin{lemma} \label{boxlemma}
Let  $\mathcal{A} \cup \mathcal{B}$ be a double $2d$ on a smooth surface $X$ of degree $d \geq  5$.  If  $\mathcal{A}$ contains three lines  $A_1$, $A_2$, $A_3$ and $\mathcal{B}$ contains five lines that all meet each $A_i$ for $i=1, 2$ and $3$, then the double $2d$ is the union of two unique quadric configurations.
\end{lemma}

\begin{proof}
Let $Q$ be the unique smooth quadric containing the three skew lines $A_1$, $A_2$, $A_3$.  Let 
$B_1$, $B_2$, $B_3$, $B_4$, and $B_5 \in \mathcal{B}$ be the five lines meeting each of $A_1$, $A_2$, $A_3$.  Since each $B_i$ meets $Q$ in three points---namely $B_i\cap A_1, B_i\cap A_2, $ and $B_i\cap A_3$---$B_i$  lies on  $Q$ for $i=1, 2, \dots, 5$.

Label the lines in $\mathcal A$ so that $A_i$ lies on $Q$ if and only if  $i\leq t$. 
We first show that $t\geq 5$.
For any $A\in \mathcal A$, note that $A$ meets all five $\{B_1, \dots, B_5\}$ if $A$ lies on $Q$ and at most two of $\{B_1, \dots, B_5\}$ if $A$ is not on $Q$. So 
$$\sum_{i = 1}^{2d} A_i\cdot\left(\sum_{j=1}^5 B_j\right)\leq  5t+2(2d-t) = 3t+4d. $$
On the other hand, each $B_j$ must intersect exactly $d+2$ lines in $\mathcal A$, so 
$$\sum_{i = 1}^{2d} A_i\cdot\left(\sum_{j=1}^5 B_j\right)= 5(d+2) = 5d+10.$$
Thus, $3t+4d \geq 5d+10$ so $t \geq 3+\frac{d+1}{3}$.
Since $d\geq 5$, we get $t\geq 5$.
Thus, the double $2d$ must contain at least five lines of  each  ruling of $Q$.

Next we show that in fact each set $\mathcal A$ and $\mathcal B$ contains $d$ lines on $Q$.
If not, let $k$ be maximal such that $A_1, \ldots, A_k$ and  $B_1, \ldots, B_k$ lie on $Q$, 
and assume without loss of generality that $A_{k+1}\not\subset Q$.
Since each $B_j$ intersects exactly $d+2$ lines in $\mathcal A$,\  $k$ of which are $A_1, \ldots, A_{k}$, we have that 
\begin{equation}\label{EQ1}
\sum_{j=1}^k B_j \cdot \left(\sum_{i=k+1}^{2d}  A_i \right) = k(d+2)-k^2 = k(d+2-k).
\end{equation}
 Since $A_i$ lies on $Q$ if and only if $i\leq k$, each line $A_j$ for $j>k$ can meet at most two of 
 $\{B_1,\dots, B_k\}$. So 
\begin{equation}\label{EQ2}
\sum_{j=1}^k B_j \cdot \left(\sum_{i=k+1}^{2d}  A_i \right) \leq 2(2d-k) = 4d-2k.
\end{equation}
Comparing (\ref{EQ1}) and (\ref{EQ2}), we have 
 $(4d-2k) - k(d+2-k) =  (k-d)(k-4) \geq  0$, which (since we've shown $k\geq 5$)  is  a contradiction unless $k\geq d$. 
We conclude that 
$\mathcal L = \{A_1, \dots, A_d\}$ and $\mathcal M = \{B_1, \dots, B_d\}$  lie on the same quadric $Q$. That is,  half the lines of $\mathcal A$, together with half the lines of $\mathcal B$,  form a quadric configuration 
$\mathcal L \cup \mathcal M$
on $X$.

It suffices to show that the remaining halves of $\mathcal A$ and $\mathcal B$  also form a quadric configuration.
 Now fix  $B \in \mathcal{B} \setminus \mathcal{M}$. Because $B$ intersects at most two of the lines of $\mathcal L$,  we know $B$  intersects every line in $\mathcal{A} \setminus \mathcal{L}$. So each line of the
skew set $\mathcal{B} \setminus \mathcal{M}$ meets every line of the skew set 
 $\mathcal{A} \setminus \mathcal{L}$---that is $\mathcal{A} \setminus \mathcal{L}$ and $ \mathcal{B} \setminus \mathcal{M}$ form a quadric configuration, as desired.
 \end{proof}
 
 \begin{remark}\label{sublemma} Even if $d$ is three or four, the final paragraph of the proof of Lemma~\ref{boxlemma} shows that  if half the lines of a double $2d$  lie on some quadric $Q$, then the complementary half lies on some other quadric $Q'$. The quadrics $\{Q, Q'\}$ may not be the only quadrics determining the double $2d$ in this case, however. See Remark~\ref{NotUnique} and Remark~\ref{double2d-not-star-line}. \end{remark}

We are now ready to prove Theorem~\ref{ProgressConj}.

\begin{proof}[Proof of Theorem~\ref{ProgressConj}]
Let $X$ be a  smooth surface of degree $d$. The case where $d=3$ is dealt with in \S~\ref{Cubics}.  We 
next handle the case $d \geq 11$.

Let  $\mathcal{A}:= \{A_1,\ldots,A_{2d}\}$ and $\mathcal{B}:= \{B_1,\ldots,B_{2d}\}$ denote the two skew sets of the double $2d$.
By Lemma~\ref{boxlemma}, it suffices to 
show that there are three skew lines in $\mathcal A$ all intersecting each of five skew lines in $\mathcal B$.
Let $M$ denote the intersection matrix $M_{ij} = A_i\cdot B_j$.
By definition of a double $2d$, $M$ has exactly $d+2$ ones and exactly $d-2$ zeros in every row and column.

  For any subset $S \subset \mathcal{A}$, let
\[
\mathrm{IntersectionSet}(S,\mathcal{B}) := \{B_i \in \mathcal{B} \mid B_i \cdot A_j = 1 \text{ for all } A_j \in S\}.
\]
We want to show that there exists some $A_i,A_j,A_k$ such that $|\mathrm{IntersectionSet}(\{A_i,A_j,A_k\}, \mathcal{B})| \geq 5$.  
After a possible relabeling of the $B_i$, we may assume
\[
\mathrm{IntersectionSet}(A_1, \mathcal{B}) = \{B_1,\dots B_{d+2}\}.
\]
Let \[k := \max_{2\leq i\leq 2d}\{|\mathrm{IntersectionSet}(\{A_1,A_i\}, \mathcal{B})|\}.\]
 Then by assumption,  the number of ones in rows $2, \dots, 2d$ and columns $1, \dots, d+2$ of $M$ is at most $k(2d-1)$ since there are at most $k$ ones in each of these rows.  However, by looking at columns, we see that there are exactly $(d+1)(d+2)$ ones in this submatrix of $M$.  Thus, we see
 \[
\frac{(d+1)(d+2)}{(2d-1)} \leq k
 \]
 and since we may assume $k$ is an integer, we have
\[
\left \lceil \frac{(d+1)(d+2)}{(2d-1)} \right \rceil \leq k.
\]
By relabelling $A_2, \dots, A_{2d}$, we may assume $|\mathrm{IntersectionSet}(\{A_1,A_2\},\mathcal{B})| = k$.  By relabelling $B_1, \dots, B_{d+2}$, we may assume $\mathrm{IntersectionSet}(\{A_1,A_2\},\mathcal{B}) = \{B_1,\dots, B_k\}$.  Now note that $M$ has exactly $kd$ ones in columns $1,\dots, k$ and rows $3, \dots, 2d$.

Let \[\ell := \max_{3\leq i\leq 2d}\{|\mathrm{IntersectionSet}(\{A_1,A_2,A_i\},\mathcal{B})|\}.\]

Then the number of ones in columns $1,\dots, k$ and rows $3, \dots, 2d$ is at most $\ell(2d-2)$, so we have
\[
\ell \geq \frac{kd}{2d-2} \geq \left \lceil \frac{(d+1)(d+2)}{(2d-1)} \right \rceil \frac{d}{2d-2}
\]
and since $\ell$ is an integer, we have
\begin{equation}\label{ellbound}
\ell \geq \left\lceil \left \lceil \frac{(d+1)(d+2)}{(2d-1)} \right \rceil \frac{d}{2d-2} \right\rceil
\end{equation}
From (\ref{ellbound}), it follows that when $d \geq 11$, we have $\ell \geq 5$, as desired.  

Finally, when $d= 4$, formula~\eqref{ellbound}  implies that  $\ell \geq 4$, so that there exists a set of three skew lines $A_1$, $A_2$, and $A_3 \in \mathcal{A}$ that all intersect four skew lines $B_1$, $B_2$, $B_3$, and $B_4 \in \mathcal{B}$.
In particular, each of the four  $B_i$ must lie on the unique quadric $Q$ determined by $A_1$, $A_2$, and $A_3$, since they each intersect this quadric in $3$ points.  We claim one more line in $\mathcal A$ lies on $Q$, in which case it follows that every double eight is the union of two quadric configurations (Remark~\ref{sublemma}). To verify the claim, observe that if no $A_i$ lies on $Q$ for $i>3$, then these
 $A_i$ intersect at most two of $B_1$, $B_2$, $B_3$, and $B_4 \in \mathcal{B}$. Thus 
$$
\sum_{i=1}^4 B_i \cdot \sum_{A_i\in \mathcal A} A_i \ =  \ \sum_{i=1}^4 B_i \cdot  \sum_{i=1}^3 A_i \  + \ \sum_{i=1}^4 B_i  \cdot \sum_{i=4}^8 A_i  \ \leq \ 12 + 10 \  =  \ 22,
$$ 
contrary to  the fact that  $\sum_{i=1}^4 B_i \cdot \sum_{A_i\in \mathcal A} A_i = 24,$ 
 since each line in $\mathcal B$ intersects exactly six lines in $\mathcal A$.

\end{proof}





\bibliographystyle{amsalpha}
\bibliography{bibdatabase}

\newcommand{\etalchar}[1]{$^{#1}$}
\providecommand{\bysame}{\leavevmode\hbox to3em{\hrulefill}\thinspace}
\providecommand{\MR}{\relax\ifhmode\unskip\space\fi MR }
\providecommand{\MRhref}[2]{%
  \href{http://www.ams.org/mathscinet-getitem?mr=#1}{#2}
}
\providecommand{\href}[2]{#2}
\begin{thebibliography}{KKP{\etalchar{+}}21b}

\bibitem[BC66]{BC}
R.~C. Bose and I.~M. Chakravarti, \emph{Hermitian varieties in a finite
  projective space {PG}(n, $q^2$)}, Canad. J. Math. \textbf{18} (1966),
  1161--1182.

\bibitem[BD12]{Bamberg+Durante}
John Bamberg and Nicola Durante, \emph{Low dimensional models of the finite
  split {C}ayley hexagon}, Theory and applications of finite fields, Contemp.
  Math., vol. 579, Amer. Math. Soc., Providence, RI, 2012, pp.~1--19.
  \MR{2964273}

\bibitem[Bea90]{beauville}
A.~Beauville, \emph{Sur les hypersurfaces dont les sections hyperplanes sont
  \`a module constant}, The {G}rothendieck {F}estschrift, {V}ol. {I}, Progr.
  Math., vol.~86, Birkh\"{a}user Boston, Boston, MA, 1990, With an appendix by
  David Eisenbud and Craig Huneke, pp.~121--133. \MR{1086884}

\bibitem[BFS13]{benito+faber+smith.measuring_sing_with_frob}
A.~Benito, E.~Faber, and K.~E. Smith, \emph{Measuring singularities with
  {F}robenius: The basics}, Commutative Algebra---{E}xpository papers dedicated
  to {D}avid {E}isenbud on the occasion of his 65th birthday (I.~Peeva, ed.),
  Springer, 2013, pp.~57--97.

\bibitem[BMS08]{blickle+mustata+smith.discr_rat_FPTs}
M.~Blickle, M.~Musta{\c{t}}\u{a}, and K.~E. Smith, \emph{Discreteness and
  rationality of {$F$}-thresholds}, Michigan Math.\ J. \textbf{57} (2008),
  43--61.

\bibitem[BR20]{Bauer+Rams}
Thomas Bauer and S{\l}awomir Rams, \emph{Counting lines on projective
  surfaces}, preprint,
  \href{https://arxiv.org/abs/1902.05133}{arXiv:1902.05133}, 2020.

\bibitem[BS07]{Boissiere+Sarti}
Samuel Boissi\`ere and Alessandra Sarti, \emph{Counting lines on surfaces},
  Ann. Sc. Norm. Super. Pisa Cl. Sci. (5) \textbf{6} (2007), no.~1, 39--52.
  \MR{2341513}

\bibitem[CC10]{Cools+Coppens}
Filip Cools and Marc Coppens, \emph{Star points on smooth hypersurfaces}, J.
  Algebra \textbf{323} (2010), no.~1, 261--286. \MR{2564838}

\bibitem[DD19]{dolgachev-duncan.automorphisms}
I.~Dolgachev and A.~Duncan, \emph{Automorphisms of cubic surfaces in positive
  characteristic}, Izv. Ross. Akad. Nauk Ser. Mat. \textbf{83} (2019), no.~3,
  15--92. \MR{3954305}

\bibitem[Dic58]{Dickson}
Leonard~Eugene Dickson, \emph{Linear groups: {W}ith an exposition of the
  {G}alois field theory}, Dover Publications, Inc., New York, 1958, With an
  introduction by W. Magnus. \MR{0104735}

\bibitem[ES16]{Etzion+Storme}
T.~Etzion and L.~Storme, \emph{Galois geometries and coding theory}, Des. Codes
  Cryptogr. \textbf{78} (2016), no.~1, 311--350. \MR{3440233}

\bibitem[FK83]{Faina+Korchmaros}
G.~Faina and G.~Korchm\'{a}ros, \emph{A graphic characterization of {H}ermitian
  curves}, Combinatorics '81 ({R}ome, 1981), Ann. Discrete Math., vol.~18,
  North-Holland, Amsterdam-New York, 1983, pp.~335--342. \MR{695821}

\bibitem[Gop83]{Goppa}
V~D Goppa, \emph{Algebraico-geometric codes}, Mathematics of the
  {USSR}-Izvestiya \textbf{21} (1983), no.~1, 75--91.

\bibitem[Gro02]{Grove}
Larry~C. Grove, \emph{Classical groups and geometric algebra}, Graduate Studies
  in Mathematics, vol.~39, American Mathematical Society, Providence, RI, 2002.
  \MR{1859189}

\bibitem[Har95]{Harris}
Joe Harris, \emph{Algebraic geometry}, Graduate Texts in Mathematics, vol. 133,
  Springer-Verlag, New York, 1995, A first course, Corrected reprint of the
  1992 original. \MR{1416564}

\bibitem[Har98]{hara.rational-singularities}
Nobuo Hara, \emph{A characterization of rational singularities in terms of
  injectivity of {F}robenius maps}, Amer. J. Math. \textbf{120} (1998), no.~5,
  981--996. \MR{1646049}

\bibitem[Hir85]{Hirschfeld}
J.~W.~P. Hirschfeld, \emph{Finite projective spaces of three dimensions},
  Oxford Mathematical Monographs, The Clarendon Press, Oxford University Press,
  New York, 1985, Oxford Science Publications. \MR{840877}

\bibitem[HK16]{HommaKim}
M.~Homma and S.~J. Kim, \emph{The characterization of {H}ermitian surfaces by
  the number of points}, J. Geom. \textbf{107} (2016), no.~3, 509--521.

\bibitem[HKT08]{Hirschfeld+Korchmaros+Torres}
J.~W.~P. Hirschfeld, G.~Korchm\'{a}ros, and F.~Torres, \emph{Algebraic curves
  over a finite field}, Princeton Series in Applied Mathematics, Princeton
  University Press, Princeton, NJ, 2008. \MR{2386879}

\bibitem[Hom97]{homma}
M.~Homma, \emph{A combinatorial characterization of the {F}ermat cubic surface
  in characteristic {$2$}}, Geom. Dedicata \textbf{64} (1997), no.~3, 311--318.
  \MR{1440564}

\bibitem[HT15]{Hirschfeld+Thas.15}
J.~W.~P. Hirschfeld and J.~A. Thas, \emph{Open problems in finite projective
  spaces}, Finite Fields Appl. \textbf{32} (2015), 44--81. \MR{3293405}

\bibitem[HT16]{Hirschfeld+Thas.16}
\bysame, \emph{General {G}alois geometries}, Springer Monographs in
  Mathematics, Springer, London, 2016. \MR{3445888}

\bibitem[HY03]{HaraYoshida}
N.~Hara and {K.-i.} Yoshida, \emph{A generalization of tight closure and
  multiplier ideals}, Trans.\ Amer.\ Math.\ Soc. \textbf{355} (2003), no.~8,
  3143--3174.

\bibitem[KKP{\etalchar{+}}21a]{extremal}
Z.~Kadyrsizova, J.~Kenkel, J.~Page, J.~Singh, K.~E. Smith, A.~Vraciu, and E.~E.
  Witt, \emph{Extremal singularities in positive characteristic}, preprint,
  \href{https://arxiv.org/abs/2009.13679}{arXiv:2009.13679}, 2021.

\bibitem[KKP{\etalchar{+}}21b]{cubics}
Zhibek Kadyrsizova, Jennifer Kenkel, Janet Page, Jyoti Singh, Karen~E. Smith,
  Adela Vraciu, and Emily~E. Witt, \emph{Cubic surfaces of characteristic two},
  Trans. Amer. Math. Soc. \textbf{374} (2021), no.~9, 6251--6267. \MR{4302160}

\bibitem[Kol15]{Kollar15}
J.~Koll\'{a}r, \emph{Szemer\'{e}di--{T}rotter-type theorems in dimension 3},
  Adv. Math. \textbf{271} (2015), 30--61.

\bibitem[KPS{\etalchar{+}}21]{WICA}
Z.~Kadyrsizova, J.~Page, J.~Singh, K.~E. Smith, A.~Vraciu, and E.~E. Witt,
  \emph{Classification of frobenius forms in dimension five}, to appear in the
  Proceedings for the 2019 Women in Commutative Algebra Workshop at the Banff
  International Research Station (2021).

\bibitem[KS11]{Klein+Storme}
A.~Klein and L.~Storme, \emph{Applications of finite geometry in coding theory
  and cryptography}, Information security, coding theory and related
  combinatorics, NATO Sci. Peace Secur. Ser. D Inf. Commun. Secur., vol.~29,
  IOS, Amsterdam, 2011, pp.~38--58. \MR{2963125}

\bibitem[Mas10]{Masini}
Tiziana Masini, \emph{A combinatorial characterization of the {H}ermitian
  surface}, Australas. J. Combin. \textbf{46} (2010), 101--107. \MR{2598696}

\bibitem[MTW05]{mustata+takagi+watanabe.F-thresholds}
M.~Musta{\c{t}}\u{a}, S.~Takagi, and {K.-i.} Watanabe, \emph{{F}-thresholds and
  {B}ernstein--{S}ato polynomials}, European {C}ongress of {M}athematics
  (Z{\"u}rich), Eur.\ Math.\ Soc., 2005, pp.~341--364.

\bibitem[Nag60]{CubicBlowUp}
Masayoshi Nagata, \emph{On rational surfaces. {I}. {I}rreducible curves of
  arithmetic genus {$0$} or {$1$}}, Mem. Coll. Sci. Univ. Kyoto Ser. A. Math.
  \textbf{32} (1960), 351--370. \MR{126443}

\bibitem[PT09]{Payne+Thas.09}
Stanley~E. Payne and Joseph~A. Thas, \emph{Finite generalized quadrangles},
  second ed., EMS Series of Lectures in Mathematics, European Mathematical
  Society (EMS), Z\"{u}rich, 2009. \MR{2508121}

\bibitem[RS15a]{Rams+Schutt.15-112lines}
S{\l}awomir Rams and Matthias Sch\"{u}tt, \emph{112 lines on smooth quartic
  surfaces (characteristic 3)}, Q. J. Math. \textbf{66} (2015), no.~3,
  941--951. \MR{3396099}

\bibitem[RS15b]{Rams+Schutt.15-64lines}
\bysame, \emph{64 lines on smooth quartic surfaces}, Math. Ann. \textbf{362}
  (2015), no.~1-2, 679--698. \MR{3343894}

\bibitem[Sch58]{Schlafli}
Ludwig Schl{\"a}fli, \emph{An attempt to determine the twenty-seven lines upon
  a surface of the third order, and to derive such surfaces in species, in
  reference to the reality of the lines upon the surface}, Quarterly Journal of
  Pure and Applied Mathematics \textbf{2} (1858), 55--65,110--120.

\bibitem[Seg43]{Segre.43}
B.~Segre, \emph{The maximum number of lines lying on a quartic surface}, Quart.
  J. Math. Oxford Ser. \textbf{14} (1943), 86--96. \MR{10431}

\bibitem[Seg67]{Segre.67}
Beniamino Segre, \emph{Introduction to {G}alois geometries}, Atti Accad. Naz.
  Lincei Mem. Cl. Sci. Fis. Mat. Natur. Sez. Ia (8) \textbf{8} (1967),
  133--236. \MR{238846}

\bibitem[Sha13]{Shafarevich}
Igor~R. Shafarevich, \emph{Basic algebraic geometry. 1}, third ed., Springer,
  Heidelberg, 2013, Varieties in projective space. \MR{3100243}

\bibitem[Shi88]{Shioda}
Tetsuji Shioda, \emph{Arithmetic and geometry of {F}ermat curves}, Algebraic
  {G}eometry {S}eminar ({S}ingapore, 1987), World Sci. Publishing, Singapore,
  1988, pp.~95--102. \MR{966448}

\bibitem[Tit76]{Tits.76}
J.~Tits, \emph{Non-existence de certains polygones g\'{e}n\'{e}ralis\'{e}s.
  {I}}, Invent. Math. \textbf{36} (1976), 275--284. \MR{435248}

\bibitem[Tit79]{Tits.79}
\bysame, \emph{Non-existence de certains polygones g\'{e}n\'{e}ralis\'{e}s.
  {II}}, Invent. Math. \textbf{51} (1979), no.~3, 267--269. \MR{530633}

\bibitem[TVN07]{Tsfasman+Vladut+Nogin}
Michael Tsfasman, Serge Vl\u{a}du\c{t}, and Dmitry Nogin, \emph{Algebraic
  geometric codes: basic notions}, Mathematical Surveys and Monographs, vol.
  139, American Mathematical Society, Providence, RI, 2007. \MR{2339649}

\bibitem[TVZ82]{Tsfasman+Vladut+Zink}
M.~A. Tsfasman, S.~G. Vl\u{a}du\c{t}, and Th. Zink, \emph{Modular curves,
  {S}himura curves, and {G}oppa codes, better than {V}arshamov-{G}ilbert
  bound}, Math. Nachr. \textbf{109} (1982), 21--28. \MR{705893}

\bibitem[TW04]{takagi+watanabe.F-pure_thresholds}
S.~Takagi and {K.-i.} Watanabe, \emph{On {F}-pure thresholds}, J.\ Algebra
  \textbf{282} (2004), no.~1, 278--297.

\bibitem[vM98]{VanMaldeghem}
Hendrik van Maldeghem, \emph{Generalized polygons}, Monographs in Mathematics,
  vol.~93, Birkh\"{a}user Verlag, Basel, 1998. \MR{1725957}

\bibitem[Zak93]{Zak}
F.~L. Zak, \emph{Tangents and secants of algebraic varieties}, Translations of
  Mathematical Monographs, vol. 127, American Mathematical Society, Providence,
  RI, 1993, Translated from the Russian manuscript by the author. \MR{1234494}

\end{thebibliography}

\end{document}